\documentclass[a4paper,reqno,11pt]{amsart}

\usepackage{amsmath,amsfonts,amssymb}
\usepackage{graphics,color}
\usepackage[latin1]{inputenc}
\usepackage[colorlinks = true,
            linkcolor = blue,
            urlcolor  = blue,
            citecolor = blue,
            anchorcolor = blue]{hyperref}

\newtheorem{theorem}{Theorem}[section]
\newtheorem{lemma}[theorem]{Lemma}

\makeatletter
\@namedef{subjclassname@2020}{\textup{2020} Mathematics Subject Classification}
\makeatother

\numberwithin{equation}{section}

\newtheorem{remark}[theorem]{Remark}
\usepackage{cite}
\numberwithin{equation}{section}

\allowdisplaybreaks[4]
\begin{document}
	
\title[Concentration behavior of normalized ground states]
{Concentration behavior of normalized ground states for mass critical Kirchhoff equations in bounded domains}

\author{Shubin Yu}
\address{School of Mathematics and Statistics, Southwest University, Chongqing 400715, People's Republic of China.}
\email{yshubin168@163.com}

\author{Chen Yang}
\address{School of Mathematics and Statistics, Southwest University, Chongqing 400715, People's Republic of China.}
\email{yangchen6858@163.com}

\author{Chun-Lei Tang$^*$}
\address{School of Mathematics and Statistics, Southwest University, Chongqing 400715, People's Republic of China.}
\email{tangcl@swu.edu.cn}

\footnotetext[1]{Corresponding author.}
\footnotetext[2]{Project supported by the National Natural Science Foundation of China (No.12371120) and Southwest University graduate research innovation project (No. SWUB24031).}



\begin{abstract}
In present paper, we study the
limit behavior of normalized ground states for the following mass critical Kirchhoff equation
$$
\left\{\begin{array}{ll}
-(a+b\int_{\Omega}|\nabla u|^2\mathrm{d}x)\Delta u+V(x)u=\mu u+\beta^*|u|^{\frac{8}{3}}u  &\mbox{in}\ {\Omega}, \\[0.1cm]
 u=0&\mbox{on}\ {\partial\Omega}, \\[0.1cm]
\int_{\Omega}|u|^2\mathrm{d}x=1, \\[0.1cm]
\end{array}
\right.
$$
where $a\geq0$, $b>0$, the function $V(x)$ is a trapping potential in a bounded domain $\Omega\subset\mathbb R^3$, $\beta^*:=\frac{b}{2}|Q|_2^{\frac{8}{3}}$ and $Q$ is the unique positive radially symmetric solution of equation
$-2\Delta u+\frac{1}{3}u-|u|^{\frac{8}{3}}u=0.$
We consider the existence of constraint minimizers for the associated energy functional involving the parameter $a$. The  minimizer corresponds to the normalized ground state of above problem, and it exists if and only if $a>0$. Moreover, when $V(x)$ attains its flattest global minimum at an inner point or only at the boundary of $\Omega$, we analyze the fine limit profiles of the minimizers as $a\searrow 0$, including mass concentration at an inner point or near the boundary of $\Omega$. In particular, we further establish the local uniqueness of the minimizer if it is concentrated at a unique inner point.

\end{abstract}
\maketitle
\vspace {-1cm}
\section{Introduction}

In present paper, we consider the limit behavior of normalized ground states for the following mass critical Kirchhoff equation
\begin{equation}\label{eqn:this-article-mass-critical-Kirchhoff-equation}
\left\{\begin{array}{ll}
-(a+b\int_{\Omega}|\nabla u|^2\mathrm{d}x)\Delta u+V(x)u=\mu u+\beta^*|u|^{\frac{8}{3}}u  &\mbox{in}\ {\Omega}, \\[0.1cm]
 u=0&\mbox{on}\ {\partial\Omega}, \\[0.1cm]
\int_{\Omega}|u|^2\mathrm{d}x=1, \\[0.1cm]
\end{array}
\right.
\end{equation}
where $a\geq0$, $b>0$, $\Omega\subset\mathbb R^3$ is a bounded domain, $\mu$ is a unknown Lagrange multiplier and  $\beta^*>0$ is a constant that will be given later.
When $V=0$, it is well known that the equation \eqref{eqn:this-article-mass-critical-Kirchhoff-equation} is an analogy of steady-state classical Kirchhoff equation, which
 was proposed by Kirchhoff \cite{Kirchhoff-1883} as an extension of the classical D'Alembert's wave equations for free vibration of elastic strings and the model takes into account the changes in length of the string produced by transverse, see
 \cite{Arosio-1996, Cavalcanti-2001} for more related physical
backgrounds.

We mention that the problem \eqref{eqn:this-article-mass-critical-Kirchhoff-equation} is mainly inspired by the recent article \cite{Hu-Tang-2021}, where the authors considered the $L^2$-norm prescribed ground states (i.e. normalized ground states) of the mass critical Kirchhoff equation in $\mathbb R^N$ as follows
\begin{equation}\label{eqn:mass-critical-Kirchhoff-equation}
\left\{\begin{array}{ll}
-(a+b\int_{\mathbb R^N}|\nabla u|^2\mathrm{d}x)\Delta u+V(x)u=\mu u+u^{\frac{8}{N}+1}  &\mbox{in}\ {\mathbb R^N}, \\[0.1cm]
\int_{\mathbb R^N}|u|^2\mathrm{d}x=c^2, \\[0.1cm]
\end{array}
\right.
\end{equation}
where $a\geq0$, $b>0$, $N=1,2,3$ and the function $V$ is a trapping potential
satisfying the condition
 \begin{enumerate} \item [($V$)] $V(x)\in L_{loc}^\infty(\mathbb R^N)$, $\min_{x\in\mathbb R^N}V(x)=0$ and  $V(x)\rightarrow+\infty$ as $|x|\rightarrow+\infty$.
\end{enumerate}
The energy functional of \eqref{eqn:mass-critical-Kirchhoff-equation} is defined by
\begin{equation*}\label{eqn:Hu-functional}
E(u):=\int_{\mathbb{R}^{N}}(a|\nabla u|^2+V(x)u^2)\mathrm{d}x+\frac{b}{2}\left(\int_{\mathbb{R}^{N}}|\nabla u|^2\mathrm{d}x\right)^2-\frac{N}{N+4}
\int_{\mathbb{R}^{N}}|u|^{2+\frac{8}{N}}\mathrm{d}x.
\end{equation*}
Consider the constrained minimization problem
$$
e(a,c):=\inf_{\{u\in\mathcal H:\int_{\mathbb R^N}u^2\mathrm{d}x=c^2\}} E(u),
$$
where $\mathcal H:=\{u\in H^1(\mathbb R^N):\int_{\mathbb R^N}V(x)u^2\mathrm dx<\infty\}$.
As is well known, the minimizers will be the normalized ground states of problem \eqref{eqn:mass-critical-Kirchhoff-equation}.
If the set $\{x\in\mathbb R^N: V(x)=0\}\not=\emptyset$,
 Hu and Tang \cite{Hu-Tang-2021} determined a threshold $c=c_*=(\frac{b|Q|_2^{\frac{8}{N}}}{2})^{\frac{N}{8-2N}}$
such that there exists at least one minimizer for $e(0,c)$
 if $0<c<c_*$
and  there is no minimizer for $e(0,c)$ if $c\geq c_*$.
On the other hand, it has been shown in
\cite[Theorem 1.2]{Ye-2019} that for $a>0$, there exists at least one minimizer for
$e(a,c)$ if $0<c\leq c_*$. In particular,
$|Q|_2$ means the $L^2$-norm of $Q$ in $\mathbb R^N$ and
$Q$ is the unique (up to translations) positive radially symmetric
solution of the following scalar field equation
\begin{equation}\label{eqn:Q-unique-radial-N}
-2\Delta u+\frac{4-N}{N}u-u^{\frac{8}{N}+1}=0,\ u\in H^1(\mathbb{R}^{N}),
\end{equation}
see for example\cite{Kwong-1989}. Moreover, we recall from \cite{Weinstein-1983} that the following
Gagliardo-Nirenberg inequality
\begin{equation}\label{eqn:GN-type-inequality-R-N}
\int_{\mathbb R^N}|u|^{2+\frac{8}{N}}\mathrm{d}x\leq \frac{N+4}{N|Q|_2^{\frac{8}{N}}}\left(\int_{\mathbb R^N}|\nabla u|^2\mathrm{d}x\right)^2
\left(\int_{\mathbb R^N} |u|^2\mathrm{d}x\right)^{\frac{4}{N}-1},\ u\in H^1(\mathbb{R}^{N}),
\end{equation}
where equality is achieved for $u(x)=Q(x)$. This inequality implies that \eqref{eqn:mass-critical-Kirchhoff-equation} is so-called mass critical problem, which is also
 crucial for obtaining the optimal existence mentioned above.

Subsequently,
in view of the blow up analysis and optimal energy estimates, the refined limit profiles of
the  minimizer
 $u_a$ of $e(a, c)$ in the case of $c=c_*$ as $a\searrow0$ were established in \cite[Theorem 1.2]{Hu-Tang-2021} provided that
 the trapping potential $V$ has $n\geq1$ isolated
minima, and that in their vicinity $V$ behaves like a power of the distance from these points,
i.e.,
 \begin{itemize} \item [($\widetilde V$)] there are numbers $p_i>0$ and $C>0$ such that\begin{equation*}
V(x)=h(x)\prod_{i=1}^n|x-x_i|^{p_i}\ \mbox{with}\ C<h(x)<\frac{1}{C}\ \mbox{for\ all}\ x\in\mathbb R^N,
\end{equation*}
where $\lim_{x\rightarrow x_i}h(x)$ exists for all $1\leq i\leq n$.
\end{itemize}
More precisely, the minimizer
must concentrate at the flattest global minimum of $V$
 and is unique for $a>0$ sufficiently small. In
addition, it is necessary to mention that the
condition $(\widetilde V)$ was proposed by \cite{Guo-2014}, where the authors gave the
original
research of  mass concentration behavior. Indeed, they studied the two-dimensional Bose-Einstein condensates (BECs)  with attractive interactions, described by the Gross-Pitaevskii (GP) functional
\begin{equation*}\label{eqn:Hu-functional}
I_a(u):=\int_{\mathbb{R}^{2}}|\nabla u|^2+V(x)|u|^2\mathrm{d}x-\frac{a}{2}
\int_{\mathbb{R}^{2}}|u|^{4}\mathrm{d}x,
\end{equation*}
where $a>0$ describes the strength of the attractive interactions. Moreover, there exists a critical value $a^*$ such that the minimizers of this functional exist only
if the interaction strength $a$ satisfies $0<a<a^*$, which means that the existence of a critical particle
number for collapse of the BECs. Based on  this point, Guo and Seiringer \cite{Guo-2014}
expound the limit behavior of minimizers depending on the behavior of $V$ near its minima. Then they considered the trapping potential $V$ satisfying $(\widetilde V)$ and established the behavior of so-called GP minimizers close
to the critical coupling strength, see \cite[Theorem 2]{Guo-2014} and the detailed descriptions mentioned there.
In particular, the further local uniqueness and refined spike profiles of the minimizers are presented in \cite{Guo-2017-SIAM}.
For more related results of the BECs or GP functional, we refer the reader to
\cite{Luo-2019,Li-2021-JMP,Guo-2016-AIHP,Guo-2018-Nonlinearity}
and their references.
Returning to the Kirchhoff problem \eqref{eqn:mass-critical-Kirchhoff-equation}, there are still some interesting results regarding the limit behavior or concentration phenomena,
see \cite{Ye-2015} for the limit behavior of mountain pass type solutions as $c\searrow c_*$ for problem \eqref{eqn:mass-critical-Kirchhoff-equation} with $V(x)=0$, \cite{Ye-2019,Guo-CPAA-2018}
for the  concentration behavior of normalized ground states in the mass-subcritical sense,  \cite{Zhu-2021} for
problem \eqref{eqn:mass-critical-Kirchhoff-equation} with a mass-subcritical perturbation, and \cite{Hu-Lu-2023} for concentration and local uniqueness of normalized ground states as $c\nearrow c_*$ for problem \eqref{eqn:mass-critical-Kirchhoff-equation} with $a=a(x)\in C^1(\mathbb R^N)$ and $V(x)=0$.

Compared to the above arguments, we are concerned with the concentration behavior and  local uniqueness of normalized ground states for the mass critical Kirchhoff equation in bound domains, that is, our problem \eqref{eqn:this-article-mass-critical-Kirchhoff-equation}.
To the best of our knowledge, only the authors in \cite{Zhu-2023} considered problem \eqref{eqn:this-article-mass-critical-Kirchhoff-equation} with
$V(x)=0$ and a mass-subcritical perturbation. However, due to $V(x)=0$, it is not possible to analyze delicate behavior, including  mass concentration at an inner point or a boundary point, and thus one cannot get the local uniqueness by the argument in \cite{Guo-2017-SIAM}.
 To this end, we will consider the function $V(x)$ as a trapping potential in a bounded domain $\Omega$, where
$\Omega\subset \mathbb R^3$ has $C^1$ boundary and satisfies the interior ball condition in the sense that for all
 $x_0\in\partial\Omega$, there exists an open ball $B\subset\Omega$ such that $x_0\in\partial B\cap\partial\Omega$.
It can be seen that
problem \eqref{eqn:this-article-mass-critical-Kirchhoff-equation} corresponds to the bounded domain version of problem \eqref{eqn:mass-critical-Kirchhoff-equation} considered in \cite{Hu-Tang-2021} with $N=3$ (for $N<3$, the results in this paper are clearly valid). Note that  the mass $\int_{\Omega}|u|^2\mathrm dx=1$, then  it is natural to consider $$\beta^*:=\frac{b}{2}|Q|_2^{\frac{8}{3}},$$
where $Q$ is the unique positive
solution of equation \eqref{eqn:Q-unique-radial-N} with $N=3$, that is,
\begin{equation}\label{eqn:Q-unique-radial}
-2\Delta u+\frac{1}{3}u-|u|^{\frac{8}{3}}u=0.
\end{equation}
Associated to problem \eqref{eqn:this-article-mass-critical-Kirchhoff-equation}, the energy functional $E_a:H_0^1(\Omega)\rightarrow\mathbb R$ is of the form
$$
E_a(u):=a\int_{\Omega}|\nabla u|^2\mathrm{d}x+\frac{b}{2}\left(\int_\Omega|\nabla u|^2\mathrm{d}x\right)^2+\int_\Omega V(x)u^2\mathrm{d}x
-\frac{3\beta^*}{7}\int_{\Omega}|u|^{2+\frac{8}{3}}\mathrm{d}x,
$$
and then the minimization problem is as follows
$$
e(a):=\inf\limits_{\left\{u\in H_0^1(\Omega):\int_{\Omega}u^2\mathrm{d}x=1\right\}}E_a(u).
$$
First, we determine that $a^*=0$ is a threshold such that
the minimizer exists if and only if $a>0$.

\begin{theorem}\label{thm:existence-nonexistence}
Assume that $V(x)$ satisfies
\begin{itemize}
  \item [($V_1$)] $\min_{x\in\bar{\Omega}}V(x)=0$\ and\ $V(x)\in C^\alpha(\bar{\Omega})$\ for\ some\ $0<\alpha<1$.
\end{itemize}
Then
\begin{itemize}
  \item [($i$)] if $a>0$, there exists at least one positive minimizer for $e(a)$;
  \item [($ii$)] if $a=0$, there is no minimizer for $e(a)$.
\end{itemize}
Moreover, there holds $\lim_{a\searrow0}e(a)=0.$
\end{theorem}

Theorem \ref{thm:existence-nonexistence} provides a general result that is independent of the shape of trapping potential $V(x)$. Moreover, this result strictly depends on the following
Gagliardo-Nirenberg inequality in bounded domains
\begin{equation}\label{eqn:GN-type-inequality}
\int_{\Omega}|u|^{2+\frac{8}{3}}\mathrm{d}x\leq \frac{7}{{3|Q|_2^{\frac{8}{3}}}}\left(\int_\Omega|\nabla u|^2\mathrm{d}x\right)^2
\left(\int_\Omega u^2\mathrm{d}x\right)^{\frac{1}{3}},\ u\in H^1_0(\Omega).
\end{equation}
The best constant in \eqref{eqn:GN-type-inequality} is the same as in \eqref{eqn:GN-type-inequality-R-N}, but it cannot be achieved in the bounded domain $\Omega$, see \cite[Appendix A]{Noris-2014}. This fact directly implies the nonexistence of minimizers, i.e., Theorem \ref{thm:existence-nonexistence}-($ii$), if we can obtain that $e(0)=0$. However, in view of ($V_1$), $V(x)$ may vanish only on $\partial\Omega$, then the usual strategy cannot get $e(0)=0$.
In this sense,  the interior ball condition of $\Omega$ seems to be necessary, see Section \ref{sec:existence-nonexistence} for the details.

In the following,
we consider the same trapping potential $V$ as \cite{Hu-Tang-2021}, i.e., $(\widetilde V)$ holds. Naturally, the $x\in\mathbb R^N$ in condition $(\widetilde V)$ should be replaced with $x\in\bar\Omega$.
Namely,
\begin{equation}\label{eqn:V-potential-x-x-i}
V(x)=h(x)\prod_{i=1}^n|x-x_i|^{p_i}\ \mbox{with}\ C<h(x)<\frac{1}{C}\ \mbox{for\ all}\ x\in\bar\Omega,
\end{equation}
where $\lim_{x\rightarrow x_i}h(x)$ also exists for all $1\leq i\leq n$. We expect the minimizers of $e(a)$ to concentrate at the flattest minimum of $V$ as $a\searrow0$ and provide the first contribution to the fine limit profiles of normalized ground states for
mass critical nonlocal Kirchhoff problem in bounded domains.

  Set
\begin{equation}\label{eqn:p-definition}
p=\max\{p_1,p_2,\cdots,p_n\},
\end{equation}

\begin{equation}\label{eqn:k-i-lambda-idefinition}
\kappa_i:=\lim_{x\rightarrow x_i}\frac{V(x)}{|x-x_i|^p},\ \lambda_i:=\left(\frac{p\kappa_i}{2|Q|_2^2}
\int_{\mathbb{R}^{3}}|x|^pQ^2(x)\mathrm{d}x\right)^{\frac{1}{p+2}},
\end{equation}

\begin{equation}\label{eqn:k-lambda-definition}
\kappa=\min\{\kappa_1,\kappa_2,\cdots,\kappa_n\},\ \lambda=\min\{\lambda_i:x_i\in\Omega,1\leq i\leq n\}
\end{equation}
and
$$
 Z_1:=\{x_i\in\Omega:p_i=p\},\  Z_0:=\{x_i\in\partial\Omega:p_i=p\},
$$
where the  latter two sets denote the locations of the flattest global minima of $V(x)$.
It turns out that only the two cases need to be considered, that is, $ Z_1\neq\emptyset$ or $ Z_1=\emptyset$.
We also point out that this observation was first discovered by \cite{Guo-Luo-Zhang-2018}, where the authors established the refined limit behaviors of minimizers for mass critical Hartree energy functionals in bounded domains.
They cleverly utilized a Gagliardo-Nirenberg type inequality with a remainder to establish a exponential estimate and then
exclude the possibility that the mass of each minimizer is close
to the boundary of $\Omega$ if $ Z_1\neq\emptyset$, or close
to a boundary point that is not in $ Z_0$ if $ Z_1=\emptyset$.
Note that for pure power nonlinearities, this inequality has been given in \cite[Theroem 5.1]{Carlen-2014}, see also Lemma \ref{lem:GFA-Theorem-5.1} below.
Then we can also get a crucial exponential estimate provided that the maximum point $z_a$ of
the minimizer $u_a$  in $\Omega$ satisfies $z_a\rightarrow x_i$ as $a\searrow 0$ for some $x_i\in\partial\Omega$, see Lemma \ref{lem:similar-to-Guo-Prop-3.1}. By this estimate, we can determine that
the mass of minimizers for $e(a)$ as $a\searrow0$ must concentrate at an inner point located in $ Z_1$ if $ Z_1\neq\emptyset$ or at a boundary point located in $ Z_0$ if $ Z_1=\emptyset$. Moreover, due to the nonlocal characteristics of Kirchhoff type problems, some more refined analyses are necessary.
In addition, we remark that the techniques given in \cite{Luo-2019} (where the authors studied the concentration behavior of BECs in bounded domains but only considered $ Z_1\neq\emptyset$) can also be used to obtain the first one,  which does not rely on the Gagliardo-Nirenberg inequality with remainder.

Now, we present the result regarding the case of $ Z_1\neq\emptyset$. As we will see, this result
 provides a detailed description of the limit behavior and explicit rate for the minimizers of $e(a)$
as $a\searrow0$, which also implies that the mass of minimizers  must concentrate at an inner point $x_0\in Z_1$.

\begin{theorem}\label{thm:mass-concentration-inner-point}
Assume that $V(x)$ satisfies \eqref{eqn:V-potential-x-x-i} and $ Z_1\neq\emptyset$. Let $u_a$ be a positive minimizer of $e(a)$ for $a>0$. For any sequence $\{a_k\}$ satisfying $a_k\searrow0$ as $k\rightarrow\infty$, there exists a subsequence, still denoted by $\{a_k\}$, such that each $u_{a_k}$ has a unique maximum point $z_{a_k}$ satisfying $\lim_{k\rightarrow\infty}z_{a_k}=x_0\in Z_1$. In particular,
$$
\lim_{k\rightarrow\infty}\frac{|z_{a_k}-x_0|}
{a_k^{\frac{1}{p+2}}}=0
$$
and
$$
\lambda^{-\frac{3}{2}}a_k^{\frac{3}{2(p+2)}}
u_{a_k}({\lambda^{-1}}{a_k^{\frac{1}{p+2}}}
x+z_{a_k})\rightarrow\frac{Q(|x|)}
{|Q|_2}
\ \mbox{in}\ H^1(\mathbb R^3)\cap L^\infty(\mathbb R^3)\ \mbox{as}\ k\rightarrow\infty,
$$
where $Q(x)$ is the unique positive radially symmetric solution of \eqref{eqn:Q-unique-radial}.
\end{theorem}


Subsequently, we consider the case of $ Z_1=\emptyset$, which implies $ Z_0\neq\emptyset$.
In this case, one can see that the method used in Theorem \ref{thm:mass-concentration-inner-point} cannot obtain the prime upper energy estimate as $a\searrow0$.
Inspired by \cite{Guo-Luo-Zhang-2018}, a delicate test
function and the more complicated analysis will be used to overcome this difficulty. In addition,
we point out that the exponential estimate in Lemma \ref{lem:similar-to-Guo-Prop-3.1} will play
a crucial role in establishing precise estimate of Lemma \ref{lem:z-a-k-x-i-ln-varepsilon-upper-lower}. Based on this
precise estimate, we can ensure the availability of the priori limit behavior in Lemma \ref{lem:w-a-k-Q-2-converge-behavior} and obtain the desired lower energy estimate of
$e(a)$. In summary, the corresponding result will be presented below, which implies that
 the mass of minimizers for $e(a)$ must concentrate near the boundary of $\Omega$ as $a\searrow0$ provided that $Z_1=\emptyset$.

\begin{theorem}\label{thm:mass-concentration-boundary}
Suppose that $V(x)$ satisfies \eqref{eqn:V-potential-x-x-i} and $ Z_1=\emptyset$. Let $u_a$ be a positive minimizer of $e(a)$ for $a>0$. For any sequence $\{a_k\}$ satisfying $a_k\searrow0$ as $k\rightarrow\infty$, there exists a subsequence, still denoted by $\{a_k\}$, such that each $u_{a_k}$ satisfies
$$
\lim\limits_{k\rightarrow\infty}\varepsilon_{a_k}^{\frac{3}{2}}u_{a_k}(\varepsilon_{a_k}x+z_{a_k})=\frac{Q(|x|)}{|Q|_2}
\ \mbox{in}\ H^1(\mathbb R^3)\cap L^\infty(\mathbb R^3),
$$
where $z_{a_k}$ is a unique maximum of $u_{a_k}$ satisfying $\lim_{k\rightarrow\infty}z_{a_k}=x_0\in Z_0$ and
$$
|z_{a_k}-x_0|\approx\frac{p+4}{2}\varepsilon_{a_k}|\ln\varepsilon_{a_k}|\ \mbox{as}\ k\rightarrow\infty,
$$
and $\varepsilon_{a_k}>0$ satisfying
$$
\varepsilon_{a_k}\approx\Big(\frac{2^{p+1}}{p\kappa(p+4)^p}\Big)
^{\frac{1}{p+2}}(p+2)^{\frac{p}{p+2}}a_k^{\frac{1}{p+2}}
\Big(\ln\frac{1}{a_k}\Big)^{-\frac{p}{p+2}}\ \mbox{as}\ k\rightarrow\infty.
$$
\end{theorem}

Finally, we are devote to giving the local uniqueness of the minimizer $u_a$ as $a\searrow0$. As shown in \cite{Guo-2017-SIAM,Hu-Tang-2021}, we know that the local uniqueness depends on the behavior of $u_a$, especially the rate at which its maximum point $z_{a_k}$ converges to the concentration point $x_0$. Indeed, if
$$
\frac{|z_{a_k}-x_0|}
{\varepsilon_{a_k}}\rightarrow\infty\ \mbox{as}\ k\rightarrow\infty,
$$
one cannot obtain the local uniqueness using known methods.
As stated in Theorem \ref{thm:mass-concentration-boundary}, it will occur if the concentration point is at the boundary.
Therefore, it seems that we can only consider the case where the concentration point is located within the domain $\Omega$ based on the Theorem \ref{thm:mass-concentration-inner-point}.
From this perspective, we assume that $Z_1=\{x_1\}$ and $p>1$, which
 mean that $Z_1$ contains only one element and $p_1=p>1$.
Define $\varepsilon_k:={\lambda^{-1}}{a_k^{\frac{1}{p+2}}}$, then Theorem \ref{thm:mass-concentration-inner-point} implies that
 $$
\frac{z_{a_k}-x_1}
{\varepsilon_k}\rightarrow 0\ \mbox{as}\ k\rightarrow\infty.
$$
Meanwhile, we point out that $p>1$ is used to ensure the crucial nondegenerate property holds. Namely,
the unique critical point $0$ of $H(y)=\int_{\mathbb R^3}|x+y|^p Q^2(x)\mathrm dx$ is nondegenerate in the sense that
\begin{equation}\label{eqn:nondegenerate}
{\rm det}\left( \frac{\partial^2H(0)}{\partial x_j\partial x_i}\right)\neq0,\ i,j=1,2,3.
\end{equation}
Thus, if  $V(x)$ satisfies \eqref{eqn:V-potential-x-x-i} with $p_1=p>1$ and $ Z_1=\{x_1\}$,
we can obtain the local uniqueness following the scheme of \cite{Guo-2017-SIAM}.
On  the other hand, note that our problem is being discussed in bounded domains.  It is worth mentioning that the authors in \cite{Li-2021-JMP} proved local uniqueness of ground states for BECs in box-shaped traps, which also inspired us.
Certainly, due to
the presence of the nonlocal term and $N=3$,
more careful analysis in the procedure is needed. We refer the reader to Section \ref{sec:Local uniqueness} for more details. Now we provide the local uniqueness result.
\begin{theorem}\label{thm:local-uniqueness}
Assume that $V(x)$ satisfies \eqref{eqn:V-potential-x-x-i} with $p_1=p>1$ and $ Z_1=\{x_1\}$, then
 the minimizer $u_a$ of $e(a)$ is unique as $a\searrow0$.
\end{theorem}
\begin{remark}\rm
We remark an interesting observation. In \cite{Guo-2021},
Guo and Zhou studied  the following Kirchhoff type energy functional
$$
E_b(u)=\int_{\mathbb{R}^{2}}(a|\nabla u|^2+V(x)u^2)\mathrm{d}x+\frac{b}{4}\left(\int_{\mathbb{R}^{2}}|\nabla u|^2\mathrm{d}x\right)^2-\frac{\beta}{4}
\int_{\mathbb{R}^{2}}|u|^{4}\mathrm{d}x
$$
and the
constrained minimization problem
$$
e(b):=\inf_{\{u\in\mathcal H:\int_{\mathbb R^2}u^2\mathrm{d}x=1\}} E_b(u),
$$
where $ \beta>0$ is a given constant. Observing that the exponent $4$ corresponds to mass critical local equation, they provided the detailed information on the limit behavior of minimizers as $b\searrow0$.
Moreover, we note that the recent paper \cite{Guo-2024} further established the limit behavior as $(a, b)\rightarrow(0,0)$ (where $b=b(a)\rightarrow 0$ as $a\searrow0$), which is also an interesting result.
We believe that such problems can also be
considered in a bounded domain $\Omega$ similar to the arguments in this article. In fact, the relevant results will be presented in the future.
\end{remark}
The remainder of this paper is organized as follows. In Section \ref{sec:existence-nonexistence}, we obtain the existence and nonexistence of minimizers and complete the proof of Theorem \ref{thm:existence-nonexistence}. Section \ref{sec:Blow-up-analysis} gives the blow-up analysis and  a crucial exponential estimate. Sections \ref{sec:Mass-concentration-at-an-inner-point} and \ref{sec:Mass-concentration-near-the-boundary} are devoted to investigating the delicate limit behavior of minimizers, and accomplishing the proof of Theorems \ref{thm:mass-concentration-inner-point}-\ref{thm:mass-concentration-boundary}, respectively. The local
uniqueness, i.e., Theorem \ref{thm:local-uniqueness}, is established in Section \ref{sec:Local uniqueness}.

\section{Existence and nonexistence of minimizers}\label{sec:existence-nonexistence}

In this section, we prove Theorem \ref{thm:existence-nonexistence} on the existence and nonexistence of minimizers for $e(a)$.
Recall the $Q=Q(|x|)$ is the unique (up to translations) positive radially symmetric solution of \eqref{eqn:Q-unique-radial}. It follows from {{\cite[Proposotion 4.1]{Gidas-1981}}} that
\begin{equation}\label{eqn:Q-property}
Q(|x|),\ |\nabla Q(|x|)|=O(|x|^{-1}e^{-|x|})\ \mbox{as}\ |x|\rightarrow\infty.
\end{equation}
Moreover, using the Pohozaev identity, one can easy to derive that
\begin{equation}\label{eqn:Q-Pohozaev-identity}
\int_{\mathbb{R}^{3}}|\nabla Q|^2\mathrm{d}x=\int_{\mathbb{R}^{3}}|Q|^2\mathrm{d}x=\frac{3}{7}\int_{\mathbb{R}^{3}}|Q|^{2+\frac{8}{3}}\mathrm{d}x.
\end{equation}
Here we also remark that the notation $C$ in the context denotes a positive constant, which might be changed from line to line and even in the same line.

\noindent\textbf{Proof of Theorem \ref{thm:existence-nonexistence}.} We first prove $(i)$, i.e., $e(a)$ has at least one positive minimizer for all $a>0$. By Gagliardo-Nirenberg type inequality \eqref{eqn:GN-type-inequality} and the definition of $\beta^*$, we obtain
\begin{equation}\label{eqn:E-lower-bounded-0}
\begin{aligned}
E_a(u)&=a\int_{\Omega}|\nabla u|^2\mathrm{d}x+\frac{b}{2}\left(\int_\Omega|\nabla u|^2\mathrm{d}x\right)^2+\int_\Omega V(x)u^2\mathrm{d}x
-\frac{3\beta^*}{7}\int_{\Omega}|u|^{2+\frac{8}{3}}\mathrm{d}x\\
&\geq a\int_{\Omega}|\nabla u|^2\mathrm{d}x\geq aC\int_{\Omega}|u|^2\mathrm{d}x,
\end{aligned}
\end{equation}
where we have used the continuous embedding $H^1_0(\Omega)\hookrightarrow L^2(\Omega)$ in the last inequality.
This gives $e(a)>0$ and $E_a(u)$ is bounded uniformly from below for $a>0$. Take the minimizing sequence $\{u_n\}\subset H_0^1(\Omega)$ of $e(a)$, that is, $\int_{\Omega}|u_n|^2\mathrm{d}x=1$ and $\lim_{n\rightarrow\infty}E_a(u_n)=e(a)$, then it is obvious that $\{u_n\}$ is bounded in $H_0^1(\Omega)$. Combining with the compactness embedding $H^1_0(\Omega)\hookrightarrow L^q(\Omega)$ for $1\leq q<6$ and passing to a subsequence of $\{u_n\}$ if necessary, we have that
$$
u_n\rightharpoonup u\ \mbox{in}\ H^1_0(\Omega)\ \mbox{and}\ u_n\rightarrow u\ \mbox{in}\ L^q(\Omega)\ \mbox{for}\ 1\leq q<6.
$$
In view of the weak lower semi-continuity, we can conclude that
$$
e(a)=\lim\limits_{n\rightarrow\infty}E_a(u_n)\geq E(u)\geq e(a),
$$
which indicates that $E_a(u)=e(a)$. Namely, we have proved the existence of a minimizer for all $a>0$. Moreover, we know that $ E_a(u)=E(|u|)$, then it follows from the maximum principle that $|u|>0$ in $\Omega$.

In the following, we prove $(ii)$, i.e., the nonexistence of minimizers for $e(0)$.  Since $\Omega\subset \mathbb{R}^{3}$ is a bounded domain, there exists an open ball $B_{2R}(x_i)\subset\Omega$, where $x_i$ is a inner point of $\Omega$ and $R>0$ is small enough. Let $\varphi\in C_0^\infty(\mathbb{R}^{3})$ be a
nonnegative cut-off function satisfying $\varphi(x)=1$ for $|x|\leq R$ and $\varphi(x)=0$ for all $|x|\geq 2R$. Consider a trial function
\begin{equation}\label{eqn:trail-function}
u_\tau(x)=\frac{A_\tau\tau^{\frac{3}{2}}}{|Q|_2}\varphi(x-x_i)Q(\tau(x-x_i)),
\end{equation}
where $A_\tau$ is determined such that $\int_{\Omega}|u_\tau|^2\mathrm{d}x=1$. In view of the exponential decay of $Q$, we can get that
$$
\frac{1}{A_\tau^2}=\frac{1}{|Q|_2^2}\int_{\mathbb R^3}\tau^3\varphi^2(x)Q^2(\tau x)\mathrm{d}x=1-o(e^{-\tau R})\ \mbox{as}\ \tau\rightarrow\infty,
$$
then
\begin{equation}\label{eqn:A-tau-1}
1\leq A_\tau^2\leq 1+o(e^{-\tau R})\ \mbox{as}\ \tau\rightarrow\infty.
\end{equation}
Moreover, we can also compute that
$$
\int_\Omega|\nabla u_\tau|^2\mathrm{d}x=\frac{A_\tau^2\tau^2}{|Q|_2^2}\int_{\mathbb R^3}|\nabla Q|^2\mathrm{d}x+o(e^{-\tau R})
$$
and
$$
\int_\Omega|u_\tau|^{2+\frac{8}{3}}\mathrm{d}x=\frac{A_\tau^{2+\frac{8}{3}}\tau^{4}}{|Q|_2^{2+\frac{8}{3}}}\int_{\mathbb R^3}|Q|^{2+\frac{8}{3}}\mathrm{d}x+o(e^{-2\tau R}) \ \mbox{as}\ \tau\rightarrow\infty.
$$
Therefore, by \eqref{eqn:Q-Pohozaev-identity} and \eqref{eqn:A-tau-1}, we can infer that
\begin{equation}\label{eqn:nabla-Kirchhoff-term-power}
\begin{aligned}
&\quad \frac{b}{2}\left(\int_\Omega|\nabla u_\tau|^2\mathrm{d}x\right)^2-\frac{3\beta^*}{7}\int_\Omega|u_\tau|^{2+\frac{8}{3}}\mathrm{d}x\\
&=\frac{bA_\tau^4\tau^4}{2|Q|_2^4}\left(\int_{\mathbb R^3}|\nabla Q|^2\mathrm{d}x\right)^2-\frac{3\beta^*A_\tau^{2+\frac{8}{3}}\tau^{4}}{7|Q|_2^{2+\frac{8}{3}}}\int_{\mathbb R^3}|Q|^{2+\frac{8}{3}}\mathrm{d}x+o(e^{-\frac{\tau R}{2}})\\
&=\frac{b}{2}A_\tau^4\tau^4(1-A_\tau^{\frac{2}{3}})+
o(e^{-\frac{\tau R}{2}})\\
&=o(e^{-\frac{\tau R}{2}})  \ \mbox{as}\ \tau\rightarrow\infty.
\end{aligned}
\end{equation}
In addition, it is easy to check that
$$
\lim\limits_{\tau\rightarrow\infty}\int_{\Omega}V(x)u_\tau^2\mathrm{d}x=\lim\limits_{\tau\rightarrow\infty}\int_\Omega V(x)\frac{A_\tau^2\tau^3}{|Q|_2^2}\varphi^2(x-x_i)Q^2(\tau(x-x_i))\mathrm{d}x=V(x_i).
$$
In light of ($V_1$), two cases may occur.
On the one hand, there is at least one inner point $x_i\in\Omega$ such that $V(x_i)=0$. Combining the above argument and \eqref{eqn:E-lower-bounded-0}, we can get that
$$
0\leq e(0)\leq\lim\limits_{\tau\rightarrow\infty}E_0(u_\tau)=V(x_i)=0,
$$
which implies that $e(0)=0$. Suppose by contradiction that if there is a minimizer $u$ such that $E_0(u)=e(0)=0$, then by \eqref{eqn:GN-type-inequality}, we derive that
$$
\frac{b}{2}\left(\int_\Omega|\nabla u|^2\mathrm{d}x\right)^2
\leq \frac{b}{2}\left(\int_\Omega|\nabla u|^2\mathrm{d}x\right)^2+\int_\Omega V(x)u^2\mathrm{d}x
=\frac{3\beta^*}{7}\int_{\Omega}|u|^{2+\frac{8}{3}}\mathrm{d}x\leq\frac{b}{2}\left(\int_\Omega|\nabla u|^2\mathrm{d}x\right)^2
$$
which indicates that
$$
\int_\Omega V(x)u^2\mathrm{d}x=V(x_i)=0\ \mbox{and}\ \frac{b}{2}\left(\int_\Omega|\nabla u|^2\mathrm{d}x\right)^2
=\frac{3\beta^*}{7}\int_{\Omega}|u|^{2+\frac{8}{3}}\mathrm{d}x.
$$
This contradicts that the best constant ${7}/({3|Q|_2^{\frac{8}{3}}})$ of \eqref{eqn:GN-type-inequality} is not attained.
On the other hand, we assume that $V(x)$ vanishes only on $\partial\Omega$, i.e., there is $x_i\in\partial\Omega$ such that $V(x_i)=0$. Since the domain $\Omega$ satisfies the interior ball condition, there exists an inner point $x_0\in\Omega$ and $R>0$ such that $B_{2R}(x_0)\subset\Omega$ and $x_i\in\partial B_{2R}(x_0)$.
Set $R_\tau=\frac{C\ln\tau}{\tau}$ for $C>2$ and $x_\tau:=x_i-2R_\tau\vec{n}$, where $\tau>0$ and $\vec{n}$ is the outer unit normal vector of $x_i$. Then we have $x_i\in\partial B_{2R_\tau}(x_\tau)$, $B_{2R_\tau}(x_\tau)\subset B_{2R}(x_0)$, $R_\tau\rightarrow 0$ and $x_\tau\rightarrow x_i$ as $\tau\rightarrow\infty$. Set $\varphi\in C_0^\infty(\mathbb{R}^{3})$ be a
nonnegative cut-off function satisfying
 $\varphi(x)=1$ for $|x|\leq 1$ and $\varphi(x)=0$ for all $|x|\geq 2$. Choose
\begin{equation}\label{eqn:a-0-trail-function}
\phi_\tau(x)=\frac{A_\tau\tau^{\frac{3}{2}}}{|Q|_2^2}\varphi\left(\frac{x-x_\tau}{R_\tau}\right)Q(\tau(x-x_\tau)),\ x\in\Omega,
\end{equation}
where $A_\tau>0$ is also determined such that $\int_{\Omega}|\phi_\tau|^2\mathrm{d}x=1$. Similar to \eqref{eqn:A-tau-1}-\eqref{eqn:nabla-Kirchhoff-term-power}, we can get that
$$
1\leq A_\tau^2\leq 1+o(\tau^{-4})\ \mbox{as}\ \tau\rightarrow\infty,
$$
and then
$$
\begin{aligned}
&\quad \frac{b}{2}\left(\int_\Omega|\nabla \phi_\tau|^2\mathrm{d}x\right)^2-\frac{3\beta^*}{7}\int_\Omega|\phi_\tau|^{2+\frac{8}{3}}\mathrm{d}x\\
&=\frac{bA_\tau^4\tau^4}{2|Q|_2^4}\left(\int_{\mathbb R^3}|\nabla Q|^2\mathrm{d}x\right)^2-\frac{3\beta^*A_\tau^{2+\frac{8}{3}}\tau^{4}}{7|Q|_2^{2+\frac{8}{3}}}\int_{\mathbb R^3}|Q|^{2+\frac{8}{3}}\mathrm{d}x+o(1)\\
&=\frac{b}{2}A_\tau^4\tau^4(1-A_\tau^{\frac{2}{3}})
+o(1)\\
&=o(1)  \ \mbox{as}\ \tau\rightarrow\infty.
\end{aligned}
$$
It follows from $V(x)\in C^\alpha(\bar{\Omega})$ for $0<\alpha<1$ that
$$
\lim\limits_{\tau\rightarrow\infty}\int_\Omega V(x)\phi_\tau^2\mathrm{d}x=V(x_i)=0.
$$
Thus,
$$
0\leq e(0)\leq\lim\limits_{\tau\rightarrow\infty}E_0(\phi_\tau)= V(x_i)=0,
$$
which implies that $e(0)=0$. The remaining process is similar to the previous case.

Finally, we check that $\lim_{a\searrow0}e(a)=0$.
It is obvious that $\lim\inf_{a\searrow0}e(a)\geq 0$ by \eqref{eqn:E-lower-bounded-0}. Taking the same trial function \eqref{eqn:a-0-trail-function}, we can obtain that
$$
e(a)\leq a\tau^2+V(x_i)+o(1) \ \mbox{as}\ \tau\rightarrow\infty,
$$
where $x_i$ satisfies $V(x_i)=0$. Choose $\tau=a^{-\frac{1}{4}}$, then
$$
\limsup\limits_{a\searrow0}e(a)\leq a^{\frac{1}{2}}+o(1)\rightarrow0.
$$
This means $\lim_{a\searrow0}e(a)=0$.
Consequently, the proof of Theorem \ref{thm:existence-nonexistence} is complete.

\section{Blow up analysis and a pre-estimate}\label{sec:Blow-up-analysis}

In this section, we first employ the blow-up analysis to determine the following  general results related to the limit behavior for the positive minimizer $u_a$ of $e(a)$ as $a\searrow0$.
\begin{lemma}\label{lem:properties-blow-up}
Assume that $V(x)$ satisfies $(V_1)$. Let $u_a$ be a positive minimizer of $e(a)$ and $\varepsilon_a$ is defined as
$$
\varepsilon_a:=\left(\int_\Omega|\nabla u_a|^2\mathrm{d}x\right)^{-\frac{1}{2}}.
$$
There hold
\begin{itemize}
  \item [($i$)] $\lim_{a\searrow0}\varepsilon_a=0$;
  \item [($ii$)] let $z_a$ be a maximum point of $u_a$ in $\Omega$ and define
  \begin{equation}\label{eqn:Omega-omega-a-definition}
  \Omega_a=\{x\in\mathbb R^3:(\varepsilon_ax+z_a)\in\Omega\},\ \omega_a(x):=\varepsilon_a^{\frac{3}{2}}u_a(\varepsilon_ax+z_a),\ x\in\Omega_a,
  \end{equation}
  then there exist positive constant $R>0$ and $\beta>0$ such that
  \begin{equation}\label{eqn:omega-a-2-beta-strict-0}
  \liminf\limits_{a\searrow0}\int_{B_{2R}\cap\Omega_a}|\omega_a|^2\mathrm{d}x\geq\beta>0.
  \end{equation}
\end{itemize}
\end{lemma}

\begin{proof}
($i$) Suppose by contradiction that $\lim_{a\searrow0}\varepsilon_a\nrightarrow0$, then there exists a sequence $\{a_k\}$ that satisfies $a_k\searrow0$ as $k\rightarrow\infty$. Moreover, $\{u_{a_k}\}$ is bounded uniformly in $H_0^1(\Omega)$. Then, up to a subsequence, there exists a  $u_0\in H_0^1(\Omega)$ such that
$$
u_{a_k}\rightharpoonup u_0\ \mbox{in}\ H^1_0(\Omega)\ \mbox{and}\ u_{a_k}\rightarrow u_0\ \mbox{in}\ L^q(\Omega)\ \mbox{for}\ 1\leq q<6.
$$
Thus, in view of the fact $e(a)\rightarrow0$ as $a\searrow0$, we conclude that
$$
0=e(0)\leq E_0(u_0)\leq\lim\limits_{k\rightarrow\infty}E_{a_k}(u_{a_k})=\lim\limits_{k\rightarrow\infty}e(a_k)=0,
$$
which shows that $u_0$ is a minimizer of $e(0)$. This contradicts the nonexistence of minimizers for $e(0)$ (i.e. Theorem \ref{thm:existence-nonexistence}-$(ii)$) and completes the proof of $(i)$.

($ii$) According to Gagliardo-Nirenberg type inequality \eqref{eqn:GN-type-inequality}, it is clear that
$$
\begin{aligned}
0&\leq\frac{b}{2}\left(\int_\Omega|\nabla u_a|^2\mathrm{d}x\right)^2-\frac{3\beta^*}{7}\int_\Omega|u_a|^{2+\frac{8}{3}}\mathrm{d}x\\
&=\frac{b}{2}\varepsilon_a^{-4}-\frac{3\beta^*}{7}\int_\Omega|u_a|^{2+\frac{8}{3}}\mathrm{d}x\\
&\leq E_a(u_a)=e(a)\rightarrow0\ \mbox{as}\ a\searrow0,
\end{aligned}
$$
which leads to
\begin{equation*}
\varepsilon_a^4\int_\Omega|u_a|^{2+\frac{8}{3}}\mathrm{d}x
\rightarrow\frac{7b}{6\beta^*}\
\end{equation*}
and
\begin{equation*}
\int_{\Omega}V(x)u_a^2\mathrm{d}x\rightarrow0\ \mbox{as}\ a\searrow0.
\end{equation*}
It follows from the definition of $\varepsilon_a$ that
\begin{equation}\label{eqn:nabla-omega-1}
\int_{\Omega_a}|\nabla\omega_a|^2\mathrm{d}x
=\varepsilon_a^2\int_\Omega|\nabla u_a|^2\mathrm{d}x=1,
\end{equation}
$$
\int_{\Omega_a}|\omega_a|^{2+\frac{8}{3}}\mathrm{d}x\rightarrow\frac{7b}{6\beta^*}\ \mbox{as}\ a\searrow0
$$
and
\begin{equation}\label{eqn:V-u-a-estimate}
\int_{\Omega}V(x)u_a^2\mathrm{d}x=\int_{\Omega_{a}}
V(\varepsilon_{a}x+z_{a})|\omega_{a}|^2\mathrm{d}x
\rightarrow0\ \mbox{as}\ a\searrow0.
\end{equation}
Since $u_a$ be a positive minimizer of $e(a)$, then
$u_a$ satisfies \eqref{eqn:this-article-mass-critical-Kirchhoff-equation}
with a associated Lagrange multiplier $\mu_a\in\mathbb R$, that is,
$$
\left\{\begin{array}{ll}
-(a+b\int_{\Omega}|\nabla u_a|^2\mathrm{d}x)\Delta u_a+V(x)u_a=\mu_a u_a+\beta^*u_a^{1+\frac{8}{3}}  &\mbox{in}\ {\Omega}, \\[0.1cm]
 u_a=0&\mbox{on}\ {\partial\Omega}, \\[0.1cm]
\end{array}
\right.
$$
We can directly derive that
$$
\mu_a=2e(a)-a\int_{\Omega}|\nabla u_a|^2\mathrm{d}x-\frac{\beta^*}{7}\int_{\Omega}|u_a|^{2+\frac{8}{3}}\mathrm{d}x-\int_\Omega V(x)u_a^2\mathrm{d}x,
$$
which signifies that
\begin{equation}\label{eqn:mu-a-estimate}
\varepsilon_a^4\mu_a\rightarrow-\frac{b}{6}\ \mbox{as}\ a\searrow0.
\end{equation}
Furthermore, we know that that $\omega_a(x)$ defined in \eqref{eqn:Omega-omega-a-definition} satisfies the following elliptic equation
\begin{equation}\label{eqn:omega-a-partial-0-equation}
\left\{\begin{array}{ll}
-(a\varepsilon_a^2+b)\Delta \omega_a+\varepsilon_a^4V(\varepsilon_ax+z_a)\omega_a=\varepsilon_a^4\mu_a \omega_a+\beta^*\omega_a^{1+\frac{8}{3}}  &\mbox{in}\ {\Omega_a}, \\[0.1cm]
 \omega_a=0&\mbox{on}\ {\partial\Omega_a}. \\[0.1cm]
\end{array}
\right.
\end{equation}
Noting that $\omega_a(x)$ attains its local maximum at $x=0$, we have $-\triangle\omega_a(0)\geq0$. Then, by $V(x)\geq0$ and \eqref{eqn:mu-a-estimate}, we can infer that
\begin{equation}\label{eqn:omega-a-0-strict-0}
\omega_a^{\frac{8}{3}}(0)\geq\frac{b}{12\beta^*}>0\ \mbox{as}\ a\searrow0
\end{equation}
and
\begin{equation}\label{eqn:omega-a-subsolution-equation}
-(a\varepsilon_a^2+b)\triangle\omega_a\leq\beta^*
\omega_a^{1+\frac{8}{3}}\ \mbox{in}\ \Omega_a.
\end{equation}
Now, we consider the following two cases.

{\bf Case 1.} There exists a constant $R>0$ such that $B_{2R}(0)\subset\Omega_a$ as $a\searrow0$. Since $\omega_a\in H^1(\Omega_a)$ is a subsolution of \eqref{eqn:omega-a-subsolution-equation}, applying the De Giorgi-Nash-Moser theory \cite[Theorem 4.1]{Han-2011}, we then obtain
$$
\max\limits_{B_R(0)}\omega_a\leq C\left(\int_{B_{2R}(0)}|\omega_a|^2\mathrm{d}x\right)^\frac{1}{2}=C\left(\int_{B_{2R}(0)\cap\Omega_a}|\omega_a|^2
\mathrm{d}x\right)^\frac{1}{2},
$$
where $C$ is a constant depending only on the bounded of $|\omega_a|_{L^r(\Omega_a)}$ for some $r\in (4,6)$. Therefore, combining with \eqref{eqn:omega-a-0-strict-0}, there exists a positive constant $\beta$ such that
$$
\liminf\limits_{a\searrow0}\int_{B_{2R}\cap\Omega_a}|\omega_a|^2\mathrm{d}x\geq\beta>0.
$$

{\bf Case 2.} For all constant $R>0$, $B_{2R}(0)\not\subset\Omega_a$ as $a\searrow0$. In this regard, define $\tilde{\omega}_a(x)\equiv\omega_a(x)$ for $x\in\Omega_a$ and $\tilde{\omega}_a(x)\equiv0$ for
$x\in\mathbb R^3\backslash\Omega_a$. Obviously, $\tilde{\omega}_a(x)$ satisfies
$$
-(a\varepsilon_a^2+b)\triangle\tilde{\omega}_a\leq\tilde{C}(x)
\tilde{\omega}_a\
\mbox{in}\ \mathbb R^3,
$$
where $\tilde{C}(x)\equiv\beta^*\omega_a^{\frac{8}{3}}(x)$ in $\Omega_a$ and $\tilde{C}(x)\equiv0$ in
$\mathbb R^3\backslash\Omega_a$. Furthermore, for any constant $R>0$, we derive $\sup_{x\in B_R(0)}\tilde w_a(x)\geq\omega_a(0)\geq\left(\frac{b}{12\beta^*}\right)^{\frac{3}{8}}$. Similar to {\bf Case 1}, we also have
$$
\sup\limits_{B_R(0)}\tilde{\omega}_a\leq C\left(\int_{B_{2R}(0)}|\tilde{\omega}_a|^2\mathrm{d}x\right)^\frac{1}{2}=C\left(\int_{B_{2R}(0)\cap\Omega_a}|
\tilde{\omega}_a|^2\mathrm{d}x\right)^\frac{1}{2}.
$$
This implies that  the proof of ($ii$) is complete.
\end{proof}

\begin{lemma}\label{lem:w-a-k-Q-2-converge-behavior}
Assume that $(V_1)$ holds, $u_a$ be a positive minimizer of $e(a)$ and $z_a$ be a maximum point of $u_a$ in $\Omega$. For any sequence $\{a_k\}$ satisfying $a_k\searrow0$ as $k\rightarrow\infty$, there exists a subsequence, still denoted by $\{a_k\}$, such that $\lim_{k\rightarrow\infty}z_{a_k}=x_{i_0}\in\bar{\Omega}$ satisfying $V(x_{i_0})=0$. Moreover, if
$$
\Omega_0=\lim\limits_{k\rightarrow\infty}\Omega_{a_k}=\mathbb R^3,
$$
then
\begin{enumerate}
  \item [($i$)] $u_{a_k}$ has a unique maximum point $z_{a_k}$ and satisfies
\begin{equation}\label{eqn:w-a-Q}
\lim\limits_{k\rightarrow\infty}\omega_{a_k}(x)=\frac{Q(|x|)}{|Q|_2}\ \mbox{in}\ H^1(\mathbb R^3)\cap L^\infty(\mathbb R^3);
\end{equation}
  \item [($ii$)]there exist $R>0$ and $C>0$ independent of $a_k$ such that
$$
\omega_{a_k}(x),\ |\nabla\omega_{a_k}(x)|,\ |\Delta\omega_{a_k}(x)|\leq Ce^{-\frac{|x|}{4}}\  \mbox{in}\ {\Omega_{a_k}\backslash B_R(0)}\ \mbox{as}\ a_k\searrow0.
$$
\end{enumerate}

\end{lemma}
\begin{proof}
Since $\{z_{a_k}\}\subset\Omega$ and $\Omega$ is a bounded domain, the sequence $\{z_{a_k}\}$ is bounded uniformly as $a_k\searrow0$. Thus, for any sequence $\{z_{a_k}\}$, there exists a subsequence, still
denoted by  $\{z_{a_k}\}$, such that $z_{a_k}\rightarrow x_{i_0}$ as $a_k\searrow0$ for some $x_{i_0}\in\bar\Omega$. We claim that $x_{i_0}$ is the minimum point of $V(x)$, i.e., $V(x_{i_0})=0$. Indeed, if $V(x_{i_0})>0$, in view of Fatou's lemma and \eqref{eqn:omega-a-2-beta-strict-0}, we can deduce that
$$
\begin{aligned}
&\quad\liminf\limits_{k\rightarrow\infty}\int_{\Omega_{a_k}}
V(\varepsilon_{a_k}x+z_{a_k})|\omega_{a_k}|^2\mathrm{d}x&\\
&\geq\int_{\Omega_{a_k}\cap B_{2R}(0)}\lim\limits_{k\rightarrow\infty}V
(\varepsilon_{a_k}x+z_{a_k})|\omega_{a_k}|^2
\mathrm{d}x\geq\zeta>0,
\end{aligned}
$$
which contradicts \eqref{eqn:V-u-a-estimate}. Hence, the claim holds.

In the following, we prove that $(i)-(ii)$.
Recall that $\varepsilon_a^4\mu_a\rightarrow\!-\frac{b}{6}\ \mbox{as}\ a\searrow0$ and $\omega_{a_k}$ satisfies
\begin{equation}\label{eqn:omega-a-Euler-Lagrange-equation}
-(a\varepsilon_{a_k}^2+b)\Delta \omega_{a_k}+\varepsilon_{a_k}^4V(\varepsilon_{a_k}x+z_{a_k})\omega_{a_k}=\varepsilon_{a_k}^4\mu_{a_k} \omega_{a_k}+\beta^*|w_{a_k}|^{\frac{8}{3}}\omega_{a_k}\  \mbox{in}\ {\Omega_{a_k}}.
\end{equation}
From now on, when necessary we shall extend $\omega_{a_k}$ to $\mathbb R^3$ by letting $\omega_{a_k}\equiv0$ on $\mathbb R^3\backslash\Omega_{a_k}$.
Under the assumption $\Omega_0=\lim_{k\rightarrow\infty}\Omega_{a_k}=\mathbb R^3$, there exists nonnegative $\omega_0\in H^1(\mathbb R^3)$ such that $\omega_{a_k}\rightharpoonup\omega_0$ in $H^1(\mathbb R^3)$ as $k\rightarrow\infty$. Passing to the weak limit of $\omega_{a_k}$, $\omega_0$ satisfies the following equation
\begin{equation}\label{eqn:omega-0-Euler-Lagrange-equation}
-b\Delta\omega_0+\frac{b}{6}\omega_0=\beta^*|\omega_0|^{\frac{8}{3}}\omega_0\ \mbox{in}\ \mathbb R^3.
\end{equation}
It follows from \eqref{eqn:omega-a-2-beta-strict-0} that $\omega_0\neq0$ and thus $\omega_0>0$ by the strong maximum principle. Note that \eqref{eqn:omega-0-Euler-Lagrange-equation}
is equivalent to
$$
-2\Delta\omega_0+\frac{1}{3}\omega_0=|Q|_2^{\frac{8}{3}}|\omega_0|^{\frac{8}{3}}\omega_0\ \mbox{in}\ \mathbb R^3
$$
and $Q$ is the unique positive solution of \eqref{eqn:Q-unique-radial}, then
using a simple rescaling, we know that
$$
\omega_0(x)=\frac{Q(|x-y_0|)}{|Q|_2}\ \mbox{for\ some}\ y_0\in\mathbb R^3
$$
and $|\omega_0|_2^2=1$, which imply that
$$
\omega_{a_k}\rightarrow\omega_0\ \mbox{in}\ L^2(\mathbb R^3)\ \mbox{as}\ k\rightarrow\infty.
$$
In virtue of interpolation inequality and the uniform boundness of $\{\omega_{a_k}\}$ in $H^1(\mathbb R^3)$, we deduce that
$$
\omega_{a_k}\rightarrow\omega_0\ \mbox{in}\ L^q(\mathbb R^3)\ \mbox{with}\ q\in[2,6)\ \mbox{as}\ k\rightarrow\infty.
$$
Combining \eqref{eqn:omega-a-Euler-Lagrange-equation} and \eqref{eqn:omega-0-Euler-Lagrange-equation}, we conclude that $\lim_{k\rightarrow\infty}|\nabla\omega_{a_k}|_2^2=|\nabla\omega_0|_2^2$. Namely,
$$
\omega_{a_k}\rightarrow\omega_0\ \mbox{in}\ H^1(\mathbb R^3)\ \mbox{as}\ k\rightarrow\infty.
$$
Moreover, by applying De Giorgi-Nash-Moser theory, there holds for any $\xi\in\mathbb R^3$,
$$
\begin{aligned}
\max\limits_{B_1(\xi)}\omega_{a_k}&\leq C\left(\int_{B_{2}(\xi)}|\omega_{a_k}|^2\mathrm{d}x\right)^\frac{1}{2}
\\
&\leq
C\left(\int_{B_{2}(\xi)}2|\omega_0|^2+2|\omega_{a_k}-\omega_0|^2
\mathrm{d}x\right)^\frac{1}{2}\\
&\leq 2^{\frac{1}{2}}C
\left(\int_{B_{2}(\xi)}|\omega_0|^2\right)^\frac{1}{2}
+ 2^{\frac{1}{2}}C\left(\int_{\mathbb R^3}|\omega_{a_k}-\omega_0|^2
\mathrm{d}x\right)^\frac{1}{2}.
\end{aligned}
$$
Taking $|\xi|\rightarrow\infty$ in above inequality, we have that
$\omega_{a_k}$ decays uniformly to zero as $|x|\rightarrow\infty$
uniformly for large $k$.
Thus, based on $V(x)\in C^\alpha(\Omega)$ and $\Omega$ is a bounded domain in $\mathbb R^3$ with $C^1$ boundary, we can infer from \eqref{eqn:omega-a-partial-0-equation} that $\omega_{a_k}\in C_{loc}^{2,\alpha_1}(\Omega_{a_k})$ for some $\alpha_1\in(0,1)$, which indicates that
\begin{equation}\label{eqn:w-ak-w-0-C2}
\omega_{a_k}\rightarrow\omega_0\ \mbox{in}\ C_{loc}^2(\mathbb R^3)\ \mbox{as}\ k\rightarrow\infty.
\end{equation}
Note that the origin is a critical point of $\omega_{a_k}$ for all $k>0$, then it is also a critical point of $\omega_0$. Hence, we conclude  that $\omega_0$ is spherically symmetric about the origin and
\begin{equation}\label{eqn:omega-0-Q-2}
\omega_0(x)=\frac{Q(|x|)}{|Q|_2}.
\end{equation}
This implies that \eqref{eqn:w-a-Q} holds.
Furthermore,
in view of \eqref{eqn:w-ak-w-0-C2}, we know that all global maximum point of $\omega_{a_k}$ must stay in a ball $B_\delta(0)$ for some small constant $\delta>0$. Due to $\omega_0''(0)<0$, we can see that $\omega_{a_k}''(|x|)<0$ in $B_\delta(0)$ for $k$ sufficiently large. Then it follows from \cite[Lemma 4.2]{Ni-1991} that $\omega_{a_k}$ has no critical point other than the origin in $B_{\delta}(0)$ as $k>0$ large enough, which shows that the uniqueness of global maximum points for $\omega_{a_k}$. As a result, $z_{a_k}$ is the unique maximum point of $u_{a_k}$ and the proof of $(i)$ is complete.

 Finally, we check $(ii)$. Note that
$\omega_{a_k}$ decays to zero as $|x|\rightarrow\infty$
uniformly for large $k$. Then by \eqref{eqn:mu-a-estimate} and \eqref{eqn:omega-a-Euler-Lagrange-equation}, we know that there exists $R>0$ independent of $a_k$ such that
$$
-(a\varepsilon_{a_k}^2+b)\Delta \omega_{a_k}+\frac{b}{16}\omega_{a_k}\leq0\  \mbox{in}\ {\Omega_{a_k}\backslash B_R(0)}.
$$
By the comparison principle, comparing $\omega_{a_k}$ with $Ce^{-\frac{|x|}{4}}$, we get that
$$
\omega_{a_k}(x)\leq Ce^{-\frac{|x|}{4}}\  \mbox{in}\ {\Omega_{a_k}\backslash B_R(0)},
$$
where $C>0$ is independent of $a_k$.
Moreover, by the definition of $\Omega_{a_k}$, $(V_1)$ and the local elliptic estimate \cite[(3.15)]{Gilbarg-1997}, we also infer from \eqref{eqn:omega-a-Euler-Lagrange-equation}  that
$$
|\nabla\omega_{a_k}(x)|,\ |\Delta\omega_{a_k}(x)|\ \leq Ce^{-\frac{|x|}{4}}\  \mbox{in}\ {\Omega_{a_k}\backslash B_R(0)}.
$$
 The proof is complete.
\end{proof}

In view of the above lemma, we know that $\lim_{k\rightarrow\infty}z_{a_k}=x_{i_0}\in\bar{\Omega}$ satisfying $V(x_{i_0})=0$. Based on subsequent analysis, we found that if $x_{i_0}\in\partial\Omega$, some more detailed information is needed. Therefore, at the end of this section, we will present
a crucial exponential estimate, whose proof depends on the Gagliardo-Nirenberg type inequality with remainder as follows.

\begin{lemma}\label{lem:GFA-Theorem-5.1}{\rm(\!{{\cite[Theroem 5.1]{Carlen-2014}}}\rm)}
There exists a constant $C_{\frac{14}{3},3}>0$ such that
$$
\int_{\mathbb R^3}|\nabla u|^2\mathrm{d}x-\left(\int_{\mathbb R^3}|u|^{2+\frac{8}{3}}\mathrm{d}x\right)^{\frac{3}{7}}\geq-
\mathcal{C}_{\frac{14}{3},3}'
+C_{\frac{14}{3},3}\inf\limits_{\Phi\in\mathcal G}\|u-\Phi\|_{H^1(\mathbb R^3)}^2
$$
for all $u\in H^1(\mathbb R^3)$ satisfying $\int_{\mathbb R^3}|u|^2\mathrm{d}x=1$, where the set $\mathcal G$ is defined as
$$
\mathcal G=\{\sigma\bar Q(\cdot-d):d\in\mathbb R^3,\sigma=\pm1\},
$$
$$
\bar Q(x)=\left(\left(\frac{6}{7}\right)^22^{\frac{3}{2}}|Q|_2^{-4}\right)
^{\frac{3}{2}}|Q|_2^{-1}
Q\left(\left(\frac{6}{7}\right)^22^{\frac{3}{2}}|Q|_2^{-4}x\right).
$$
and
$$
\mathcal{C}_{\frac{14}{3},3}'
=\frac{8}{7}\left(\frac{6}{7}\right)^3|Q|_2^{-8}.
$$
\end{lemma}

\begin{lemma}\label{lem:similar-to-Guo-Prop-3.1}
If $(V_1)$ holds and $u_a$ be a positive minimizer of $e(a)$. Suppose that $\Omega_{a_k}$, $\omega_{a_k}$ are defined by \eqref{eqn:Omega-omega-a-definition} and $z_{a_k}\rightarrow x_i$ as $k\rightarrow\infty$ for some $x_i\in\partial\Omega$ satisfying $V(x_i)=0$. Then up to a subsequence, there exists a sequence $\{\sigma_k\}$ satisfying $\sigma_k\rightarrow$ as $k\rightarrow\infty$ such that
\begin{equation}\label{eqn:key-estimate}
\frac{\left(\int_{\Omega_{a_k}}|\omega_{a_k}|^{2+\frac{8}{3}}\mathrm{d}x\right)^3}
{\left(\int_{\Omega_{a_k}}|\nabla\omega_{a_k}|^2\mathrm{d}x\right)^6}
\leq\left(\frac{7}{3}\right)^{3}|Q|_2^{-8}- Ce^{-\frac{2}{1+\sigma_k}\cdot\frac{|z_{a_k}-x_i|}{\varepsilon_{a_k}}}
\ \mbox{as}\ k\rightarrow\infty,
\end{equation}
where $\varepsilon_{a_k}$ is as in the Lemma \rm \ref{lem:properties-blow-up}.
\end{lemma}
\begin{proof}
Define $\tilde\omega_{a_k}\equiv\omega_{a_k}$ for $x\in\Omega_{a_k}$ and $\tilde\omega_{a_k}\equiv0$ for $x\in\mathbb R^3\backslash\Omega_{a_k}$, then
$$
\int_{\mathbb R^3}|\nabla\tilde\omega_{a_k}|^2\mathrm{d}x=
\int_{\Omega_{a_k}}|\nabla\omega_{a_k}|^2\mathrm{d}x\ \mbox{and}\
\int_{\mathbb R^3}|\tilde\omega_{a_k}|^q\mathrm{d}x=\int_{\Omega_{a_k}}
|\tilde\omega_{a_k}|^q\mathrm{d}x\ \mbox{for}\ q>1.
$$
Set $\bar\omega_{a_k}=\lambda_k^{\frac{3}{2}}\tilde\omega_{a_k}(\lambda_kx)$, where
$$
\lambda_k=\left(\frac{7}{6}\right)^{-\frac{7}{2}}
\frac{\left(\int_{\mathbb R^3}|\tilde\omega_{a_k}|^{2+\frac{8}{3}}\mathrm{d}x\right)^{\frac{3}{2}}}
{\left(\int_{\mathbb R^3}|\nabla\tilde\omega_{a_k}|^2\mathrm{d}x\right)^{\frac{7}{2}}}.
$$
According to Lemma \ref{lem:GFA-Theorem-5.1}, it is not difficult to calculate that
$$
\left(\left(\frac{7}{6}\right)^{-7}-\left(\frac{7}{6}\right)^{-6}\right)
\frac{\left(\int_{\mathbb R^3}|\tilde\omega_{a_k}|^{2+\frac{8}{3}}\mathrm{d}x\right)^{3}}
{\left(\int_{\mathbb R^3}|\nabla\tilde\omega_{a_k}|^2\mathrm{d}x\right)^{6}}
\geq -\mathcal{C}_{\frac{14}{3},3}'+C_{\frac{14}{3},3}\inf\limits_{\Phi\in\mathcal G}\|\bar\omega_{a_k}-\Phi\|_{H^1(\mathbb R^3)}^2\ \mbox{as}\ k\rightarrow\infty.
$$
This above inequality gives that
\begin{equation}\label{eqn:Phi-mathcal-G-inf}
\begin{aligned}
\frac{\left(\int_{\Omega_{a_k}}|\omega_{a_k}|^{2+\frac{8}{3}}\mathrm{d}x\right)^{3}}
{\left(\int_{\Omega_{a_k}}|\nabla\omega_{a_k}|^2\mathrm{d}x\right)^{6}}
&\leq-7\left(\frac{7}{6}\right)^{6}
\left(-\mathcal{C}_{\frac{14}{3},3}'+C_{\frac{14}{3},3}\inf\limits_{\Phi\in\mathcal G}\|\bar\omega_{a_k}-\Phi\|_{H^1(\mathbb R^3)}^2\right)\\
&=\left(\frac{7}{3}\right)^{3}|Q|_2^{-8}-
7\left(\frac{7}{6}\right)^{6}C_{\frac{14}{3},3}\inf\limits_{\Phi\in\mathcal G}\|\bar\omega_{a_k}-\Phi\|_{H^1(\mathbb R^3)}^2
\ \mbox{as}\ k\rightarrow\infty.
\end{aligned}
\end{equation}
In view of $u_{a_k}\in H^1_0(\Omega)$, we have $\omega_{a_k}\in H^1_0(\Omega_{a_k})$. By Lemma \ref{lem:properties-blow-up}, we know that
$\lambda_k\rightarrow\left(\frac{6}{7}\right)^{2}2^{\frac{3}{2}}|Q|_2^{-4}$. Then there exists a sequence $\{\sigma_k\}$ satisfying $\sigma_k\rightarrow0$ as $k\rightarrow\infty$ such that $\lambda_k=(1+\sigma_k)\left(\frac{6}{7}\right)^{2}2^{\frac{3}{2}}|Q|_2^{-4}$.
Hence, from the definitions of $\bar \omega_{a_k}$ and $\mathcal G$ (see Lemma \ref{lem:GFA-Theorem-5.1}), there is a sequence $\{y_k\}\in\mathbb R^3$ such that
\begin{equation}\label{eqn:Phi-omega-a-k-2}
\begin{aligned}
\inf\limits_{\Phi\in\mathcal G}\|\bar\omega_{a_k}-\Phi\|_{H^1(\mathbb R^3)}^2
&=\|\bar\omega_{a_k}-\bar Q(\cdot-y_k)\|_{H^1(\mathbb R^3)}^2\\
&\geq\int_{\mathbb R^3}\Big|\tilde \omega_{a_k}-\frac{1}{(1+\sigma_k)^{\frac{3}{2}}|Q|_2}
Q(\frac{x-y_k}{1+\sigma_k})\Big|^2\mathrm{d}x
\ \mbox{as}\ k\rightarrow\infty.
\end{aligned}
\end{equation}

In the following, we consider the two cases. First, if the sequence $\{y_k\}$ satisfies $|\frac{y_k}{1+\sigma_k}|\rightarrow\infty$ as $k\rightarrow\infty$, we can see from \eqref{eqn:omega-a-2-beta-strict-0} and \eqref{eqn:Phi-omega-a-k-2} that
$$
\inf\limits_{\Phi\in\mathcal G}\|\bar\omega_{a_k}-\Phi\|_{H^1(\mathbb R^3)}^2\geq\frac{\beta}{2}>0\ \mbox{as}\ k\rightarrow\infty.
$$
Note that the part $e^{-\frac{2}{1+\sigma_k}\cdot\frac{|z_{a_k}-x_i|}{\varepsilon_{a_k}}}$
 is always bounded, then it follows from \eqref{eqn:Phi-mathcal-G-inf} that
 there exists  $C>0$ such that
\eqref{eqn:key-estimate} holds.
Now we discuss the other case where $|\frac{y_k}{1+\sigma_k}|\leq C_0$ holds uniformly in $k$.
Observe that $x_i\in\partial\Omega$ and then $\Omega_{a_k}^c\cap B_{\frac{|z_{a_k}-x_i|}{\varepsilon_{a_k}}+2C_0}
(\frac{y_k}{1+\sigma_k})\neq\emptyset$.
Therefore, in view of the definition of $\tilde \omega_{a_k}$, exponential decay of $Q$ and \eqref{eqn:Phi-omega-a-k-2} that
$$
\begin{aligned}
\inf\limits_{\Phi\in\mathcal G}\|\bar\omega_{a_k}-\Phi\|_{H^1(\mathbb R^3)}
&\geq\int_{\Omega_{a_k}^c}\Big|\frac{1}{(1+\sigma_k)^{\frac{3}{2}}|Q|_2}Q(\frac{x-y_k}{1+\sigma_k})\Big|^2\mathrm{d}x\\
&\geq\int_{\Omega_{a_k}^c\cap B_{\frac{|z_{a_k}-x_i|}{\varepsilon_{a_k}}+2C_0}(\frac{y_k}{1+\sigma_k})}
\frac{1}{(1+\sigma_k)^3|Q|_2^2}Q^2(\frac{x-y_k}{1+\sigma_k})\mathrm{d}x\\
&\geq Ce^{-\frac{2}{1+\sigma_k}\cdot\frac{|z_{a_k}-x_i|}{\varepsilon_{a_k}}}
\ \mbox{as}\ k\rightarrow\infty.
\end{aligned}
$$
Then by \eqref{eqn:Phi-mathcal-G-inf}, the desire result is also reached.
\end{proof}
\section{Mass concentration at an inner point}\label{sec:Mass-concentration-at-an-inner-point}

 Under the assumption that $V(x)$ satisfying \eqref{eqn:V-potential-x-x-i} and $Z_1\neq\emptyset$,
in this section we will focus on the inner mass concentration of positive minimizers as $a\searrow0$, that is,
 Theorem \ref{thm:mass-concentration-inner-point} is obtained. We start with the following energy estimate of $e(a)$ when $a$ is close to $0$.

\begin{lemma}\label{lem:e-a-upper-bound-estimate}
Suppose $V(x)$ satisfies \eqref{eqn:V-potential-x-x-i} and $Z_1\neq\emptyset$, then we have
$$
0\leq e(a)\leq \frac{p+2}{p}\lambda^2a^\frac{p}{p+2}\ \mbox{as}\ a\searrow0,
$$
where $p>0$ and $\lambda>0$ are defined by \eqref{eqn:p-definition} and \eqref{eqn:k-lambda-definition}, respectively.
\end{lemma}
\begin{proof}
Note that $ Z_1\neq\emptyset$ and $\Omega$ is a bounded domain of $\mathbb R^3$, then there exists an open ball $B_{2R}(x_i)\subset\Omega$ centered at an inner point $x_i\in Z_1$ for $i\in\{1,2,\cdots,n\}$, where $R>0$ is sufficiently small. Moreover, there also exists $M>0$ such that
$$
V(x)\leq M|x-x_i|^p\ \mbox{in}\ \Omega,
$$
which yields that
\begin{equation}\label{eqn:V-leq-M-p-x-x-i}
V(\frac{x}{\tau}+x_i)\leq M\tau^{-p}|x|^p\ \mbox{for\ any}\ x\in\Omega_\tau:=\{x\in\mathbb R^3:(\frac{x}{\tau}+x_i)\in\Omega\}.
\end{equation}
Consider a nonnegative cut-off function $\varphi\in C_0^\infty(\mathbb R^3)$ such that $\varphi(x)=1$ for $|x|\leq R$ and $\varphi(x)=0$ for $|x|\geq 2R$ and set a trial function of form \eqref{eqn:trail-function}, i.e.,
$$
u_\tau(x)=\frac{A_\tau\tau^{\frac{3}{2}}}{|Q|_2}\varphi(x-x_i)Q(\tau(x-x_i)),
$$
where $A_\tau$ is determined such that $\int_{\Omega}|u_\tau|^2\mathrm{d}x=1$. In light of the exponential decay of $Q$, we have
$$
1\leq A_\tau^2\leq1+o(e^{-\tau R})\ \mbox{as}\ \tau\rightarrow\infty.
$$
Together with \eqref{eqn:V-leq-M-p-x-x-i}, we infer that
$$
\begin{aligned}
\int_{\Omega}V(x)u_\tau^2\mathrm{d}x
&=\frac{A^2_\tau}{|Q|_2^2}\int_{\Omega_\tau}V(\frac{x}{\tau}+x_i)\varphi^2(\frac{x}{\tau})Q^2(x)\mathrm{d}x\\
&=\frac{A^2_\tau}{|Q|_2^2}\int_{|x|\leq \tau R}V(\frac{x}{\tau}+x_i)Q^2(x)\mathrm{d}x+\frac{A^2_\tau}{|Q|_2^2}\int_{\tau R\leq|x|\leq 2\tau R}V(\frac{x}{\tau}+x_i)\varphi^2(\frac{x}{\tau})Q^2(x)\mathrm{d}x\\
&=\frac{A^2_\tau}{|Q|_2^2}\int_{\mathbb R^3}V(\frac{x}{\tau}+x_i)Q^2(x)\mathrm{d}x
-\frac{A^2_\tau}{|Q|_2^2}\int_{|x|\geq\tau R}V(\frac{x}{\tau}+x_i)Q^2(x)\mathrm{d}x\\
&\quad\ +\frac{A^2_\tau}{|Q|_2^2}M\tau^{-p}\int_{\tau R\leq|x|\leq 2\tau R}|x|^pQ^2(x)\mathrm{d}x+o(\tau^{-p})\\
&=\frac{\kappa_i\tau^{-p}}{|Q|_2^2}\int_{\mathbb R^3}|x|^pQ^2(x)\mathrm{d}x+o(\tau^{-p})\ \mbox{as}\ \tau\rightarrow\infty.
\end{aligned}
$$
Furthermore, similar to \eqref{eqn:nabla-Kirchhoff-term-power}, we can conclude that
$$
\begin{aligned}
e(a)&\leq E_a(u_\tau)\\
&=\frac{aA_\tau^2\tau^2}{|Q|_2^2}\int_{\mathbb R^3}|\nabla Q|^2\mathrm{d}x+\frac{bA_\tau^4}{2|Q|_2^4}
\left(\int_{\mathbb R^3}|\nabla Q|^2\mathrm{d}x\right)^2+\int_{\Omega}V(x)u_\tau^2\mathrm{d}x\\
&\quad\ -\frac{3\beta^*A_\tau^{2+\frac{8}{3}}\tau^{4}}{7|Q|_2^{2+\frac{8}{3}}}\int_{\mathbb R^3}|Q|^{2+\frac{8}{3}}\mathrm{d}x+o(\tau^{-p})\\
&=a\tau^2+\frac{\kappa_i\tau^{-p}}{|Q|_2^2}\int_{\mathbb R^3}|x|^pQ^2(x)\mathrm{d}x+o(\tau^{-p})\ \mbox{as}\ \tau\rightarrow\infty.
\end{aligned}
$$
By choosing
$$
\tau=\left(\frac{p\kappa_i}{2|Q|_2^2}\int_{\mathbb{R}^{3}}|x|^pQ^2(x)\mathrm{d}x\right)^{\frac{1}{p+2}}a^{-\frac{1}{p+2}}
=\lambda_ia^{-\frac{1}{p+2}},\ i\in\{1,2,\cdots,n\},
$$
so that $\tau\rightarrow\infty$ as $a\searrow0$,
we derive that the upper bound of $e(a)$ is as follows
$$
e(a)\leq\frac{p+2}{p}\lambda^2a^{\frac{p}{p+2}}\ \mbox{as}\  a\searrow0.
$$
This completes the proof.
\end{proof}

\begin{lemma}\label{lem:z-a-k-mathcal-Z-1}
Suppose $V(x)$ satisfies \eqref{eqn:V-potential-x-x-i} and $Z_1\neq\emptyset$, then for the subsequence $\{a_k\}$ obtained in Lemma {\rm \ref{lem:w-a-k-Q-2-converge-behavior}}, there holds $z_{a_k}\rightarrow x_i$ for some $x_i\in Z_1$.
\end{lemma}

\begin{proof}
According to Lemma \ref{lem:w-a-k-Q-2-converge-behavior}, we know $z_{a_k}\rightarrow x_{i_0}$ as $k\rightarrow\infty$ and $V(x_{i_0})=0$. Now, we prove that $\left\{\frac{|z_{a_k}-x_{i_0}|}{\varepsilon_{a_k}}\right\}$ is bounded uniformly as $k\rightarrow\infty$. Assume to the contrary that there is a subsequence, still denoted by $\{a_k\}$,  such that $\frac{|z_{a_k}-x_{i_0}|}{\varepsilon_{a_k}}\rightarrow\infty$ as $k\rightarrow\infty$, which indicates $\Omega_0=\lim_{k\rightarrow\infty}\Omega_{a_k}=\mathbb R^3$. Taking into account that \eqref{eqn:omega-a-2-beta-strict-0}, we deduce that for any large constant $M>0$,
\begin{equation}\label{eqn:bounded-V-frac-z-a-k-x-i-varepsilon}
\begin{aligned}
&\quad\liminf\limits_{\varepsilon_{a_k}\rightarrow0}\frac{1}{\varepsilon_{a_k}^{p_{i_0}}}
\int_{\mathbb R^3}V(\varepsilon_{a_k}x+z_{a_k})|\omega_{a_k}|^2\mathrm{d}x\\
&\geq C\int_{B_{2R}(0)\cap\Omega_{a_k}}\liminf\limits_{\varepsilon_{a_k}\rightarrow0}\Big|x
+\frac{z_{a_k}-x_{i_0}}{\varepsilon_{a_k}}\Big|^{p_{i_0}}
\prod\limits_{j=1,j\neq i_0}^n|\varepsilon_{a_k}x+z_{a_k}-x_j|^{p_j}|\omega_{a_k}|^2\mathrm{d}x\\
&\geq M,
\end{aligned}
\end{equation}
where $p_{i_0}\in(0,p]$ is the exponent corresponding to $x_{i_0}$ in \eqref{eqn:V-potential-x-x-i}.
Thus it follows from \eqref{eqn:GN-type-inequality} that
\begin{equation}\label{eqn:contrary-bounded-e-a-V-M}
\begin{aligned}
e(a_k)&\geq a_k\varepsilon_{a_k}^{-2}+M\varepsilon_{a_k}^{p_{i_0}}\\
&\geq
\Big(\left(\frac{p_{i_0}}{2}\right)^{\frac{2}{p_{i_0}+2}}+\left(\frac{2}{p_{i_0}}\right)^{\frac{p_{i_0}}{p_{i_0}+2}}\Big)
M^{\frac{2}{p_{i_0}+2}}a_k^{\frac{p_{i_0}}{p_{i_0}+2}}\ \mbox{as}\ a_k\searrow0,
\end{aligned}
\end{equation}
which leads to a contradiction with Lemma \ref{lem:e-a-upper-bound-estimate} since $M$ is arbitrary large.

In light of the above argument, we can exclude the following two cases.
 On the one hand, it is impossible that $z_{a_k}\rightarrow x_{i_0}$ for some $x_{i_0}\in\Omega$ and $p_{i_0}<p$.
If this case  occurs, it follows from $x_{i_0}$ is an inner point of $\Omega$ that $\Omega_0=\lim\limits_{k\rightarrow\infty}\Omega_{a_k}=\mathbb R^3$. Then Lemma \ref{lem:w-a-k-Q-2-converge-behavior} implies that
$$
\lim\limits_{k\rightarrow\infty}\omega_{a_k}(x)=\frac{Q(|x|)}{|Q|_2}\ \mbox{in}\ H^1(\mathbb R^3).
$$
Since $\left\{\frac{|z_{a_k}-x_{i_0}|}{\varepsilon_{a_k}}\right\}$ is bounded uniformly as $k\rightarrow\infty$, we can deduce that for  $R>0$ small enough, there exists $C_0(R)>0$, independent of $a_k$, such that
$$
\begin{aligned}
&\quad\liminf\limits_{\varepsilon_{a_k}\rightarrow0}\frac{1}{\varepsilon_{a_k}^{p_{i_0}}}
\int_{\mathbb R^3}V(\varepsilon_{a_k}x+z_{a_k})|\omega_{a_k}|^2\mathrm{d}x\\
&\geq C\int_{B_{R}(0)}\liminf\limits_{\varepsilon_{a_k}\rightarrow0}\Big|x+\frac{z_{a_k}-x_{i_0}}{\varepsilon_{a_k}}\Big|^{p_{i_0}}
\prod\limits_{j=1,j\neq i_0}^n|\varepsilon_{a_k}x+z_{a_k}-x_j|^{p_j}|\omega_{a_k}|^2\mathrm{d}x\\
&\geq C_0(R),
\end{aligned}
$$
which gives that
$$
e(a_k)\geq a_k\varepsilon_{a_k}^{-2}+C_0(R)\varepsilon_{a_k}^{p_{i_0}}
\geq\Big(\left(\frac{p_{i_0}}{2}\right)^{\frac{2}{p_{i_0}+2}}+\left(\frac{2}{p_{i_0}}\right)^{\frac{p_{i_0}}{p_{i_0}+2}}\Big)
(C_0(R))^{\frac{2}{p_{i_0}+2}}a_k^{\frac{p_{i_0}}{p_{i_0}+2}}\ \mbox{as}\ a_k\searrow0.
$$
This contradicts Lemma \ref{lem:e-a-upper-bound-estimate}.

On the other hand,  it should be rule out the possibility that $z_{a_k}\rightarrow x_{i_0}$ as $k\rightarrow\infty$ for some $x_{i_0}\in\partial\Omega$.
Assuming this possibility arises, we can infer from Lemma \ref{lem:similar-to-Guo-Prop-3.1} that there exists a sequence $\{\sigma_k\}$
satisfying $\sigma_k\rightarrow0$ as $k\rightarrow\infty$ such that
$$
\frac{\left(\int_{\Omega_{a_k}}|\omega_{a_k}|^{2+\frac{8}{3}}\mathrm{d}x\right)^3}
{\left(\int_{\Omega_{a_k}}|\nabla\omega_{a_k}|^2\mathrm{d}x\right)^6}
\leq\left(\frac{7}{3}\right)^{3}|Q|_2^{-8}- Ce^{-\frac{2}{1+\sigma_k}\cdot\frac{|z_{a_k}-x_i|}{\varepsilon_{a_k}}}
\ \mbox{as}\ k\rightarrow\infty.
$$
According to $\int_{\Omega_{a_k}}|\nabla\omega_{a_k}|^2\mathrm{d}x=1,$ $\varepsilon_{a_k}^{-4}\rightarrow\infty$ and $\left\{\frac{|z_{a_k}-x_{i_0}|}{\varepsilon_{a_k}}\right\}$ is bounded uniformly as $k\rightarrow\infty$, we deduce that
\begin{equation}\label{eqn:using-guo-prop-3.1-e-a}\small
\begin{aligned}
e(a_k)=E_a(u_{a_k})
&\geq\frac{a_k}{\varepsilon_{a_k}^2}\int_{\Omega_{a_k}}|\nabla\omega_{a_k}|^2\mathrm{d}x
+\frac{b}{2\varepsilon_{a_k}^4}\left(\int_{\Omega_{a_k}}|\nabla\omega_{a_k}|^2\mathrm{d}x\right)^2
-\frac{3\beta^*}{7\varepsilon_{a_k}^4}\int_{\Omega_{a_k}}|\omega_{a_k}|^{2+\frac{8}{3}}\mathrm{d}x\\
&\geq\frac{a_k}{\varepsilon_{a_k}^2}
+\frac{1}{\varepsilon_{a_k}^4}\left[\frac{b}{2}
-\frac{3\beta^*}{7}\left(\left(\frac{7}{3}\right)^3|Q|_2^{-8}- Ce^{-\frac{2}{1+\sigma_k}\frac{|z_{a_k}-x_{i_0}|}{\varepsilon_{a_k}}}
\right)^{\frac{1}{3}}\right]\\
&\geq C\varepsilon_{a_k}^{-4}e^{-\frac{2}{1+\sigma_k}\frac{|z_{a_k}-x_{i_0}|}{\varepsilon_{a_k}}}\rightarrow\infty
\ \mbox{as}\ k\rightarrow\infty,
\end{aligned}
\end{equation}
where we have used the definition of $\beta^*$ and the inequality $(1-t)^{\frac{1}{3}}\leq1-\frac{1}{3}t$ for $0\leq t\leq 1$. This is a contradiction with Lemma \ref{lem:e-a-upper-bound-estimate}. Consequently, we have that $z_{a_k}\rightarrow x_{i}$ as $k\rightarrow\infty$ for some $x_{i}\in Z_1$ and the proof is complete.
\end{proof}

\noindent\textbf{Proof of Theorem \ref{thm:mass-concentration-inner-point}.} In view of the proof of Lemma \ref{lem:z-a-k-mathcal-Z-1}, we know that $\left\{\frac{|z_{a_k}-x_i|}{\varepsilon_{a_k}}\right\}$ is bounded uniformly as $k\rightarrow\infty$, then we can suppose that $\frac{z_{a_k}-x_i}{\varepsilon_{a_k}}\rightarrow y_0$ as $k\rightarrow\infty$ form some $y_0\in\mathbb R^3$, where $x_i\in Z_1$. Note that $ Z_1\subset\Omega$, then $\Omega_0=\lim_{k\rightarrow\infty}\Omega_{a_k}=\mathbb R^3$. Thus, by Fatou's lemma and \eqref{eqn:omega-0-Q-2}, we deduce that
\begin{equation}\label{eqn:x-y-0-p-V}
\begin{aligned}
&\quad\ \liminf\limits_{\varepsilon_{a_k}\rightarrow0}\frac{1}{\varepsilon_{a_k}^p}
\int_{\Omega_{a_k}}V(\varepsilon_{a_k}x+z_{a_k})\omega_{a_k}^2\mathrm{d}x\\
&\geq \kappa_i\int_{\mathbb R^3}|x+y_0|^p\omega_0^2\mathrm{d}x\geq\kappa_i\int_{\mathbb R^3}|x|^p\omega_0^2\mathrm{d}x
=\frac{\kappa_i}{|Q|_2^2}\int_{\mathbb R^3}|x|^pQ^2(|x|)\mathrm{d}x,
\end{aligned}
\end{equation}
where $\kappa_i$ is defined by \eqref{eqn:k-i-lambda-idefinition}. Then we get that
$$
\begin{aligned}
\liminf\limits_{a_k\searrow0}e(a_k)
&\geq\frac{a_k}{\varepsilon^2_{a_k}}
+\int_{\Omega_{a_k}}V(\varepsilon_{a_k}x+z_{a_k})\omega_{a_k}^2\mathrm{d}x\\
&\geq\frac{a_k}{\varepsilon_{a_k}^2}+\varepsilon_{a_k}^p\frac{\kappa_i}{|Q|_2^2}\int_{\mathbb R^3}|x|^pQ^2(|x|)\mathrm{d}x\\
&=\frac{a_k}{\varepsilon_{a_k}^2}+\frac{2}{p}\lambda_i^{p+2}\varepsilon_{a_k}^p
\geq\frac{p+2}{p}\lambda_i^2a_k^{\frac{p}{p+2}},
\end{aligned}
$$
where $i\in\{1,2,\cdots,n\}$. Naturally,
$$
\liminf\limits_{a_k\searrow0}e(a_k)\geq\frac{p+2}{p}\lambda^2a_k^{\frac{p}{p+2}},
$$
which, together with Lemma \ref{lem:e-a-upper-bound-estimate} and \eqref{eqn:x-y-0-p-V}, implies that
$$
\lim\limits_{k\rightarrow\infty}\frac{e(a_k)}{a_k^{\frac{p}{p+2}}}
=\frac{p+2}{p}\lambda^2,\ \
\lim\limits_{k\rightarrow\infty}\frac{\varepsilon_{a_k}}{(a_k)
^{\frac{1}{p+2}}}=\frac{1}{\lambda}
$$
and $y_0=0$, i.e., $\lim_{k\rightarrow\infty}\frac{|z_{a_k}-x_i|}{\varepsilon_{a_k}}=0$. Thus, combining with Lemma \ref{lem:w-a-k-Q-2-converge-behavior},
we can conclude that
$$
\lambda^{-\frac{3}{2}}a_k^{\frac{3}{2(p+2)}}u_{a_k}({\lambda^{-1}}a_k^{\frac{1}{p+2}}x
+z_{a_k})\rightarrow\frac{Q(|x|)}{|Q|_2}
\ \mbox{in}\ H^1(\mathbb R^3)\ \mbox{as}\ k\rightarrow\infty,
$$
where $z_{a_k}\rightarrow x_i\in Z_1$ as $k\rightarrow\infty$. The proof of Theorem \ref{thm:mass-concentration-inner-point} is complete.

\section{Mass concentration at a boundary point}\label{sec:Mass-concentration-near-the-boundary}

In this section, we consider the case of $ Z_1=\emptyset$.
We determine that the positive minimizers for $e(a)$ will concentrate near the boundary of $\Omega$ and accomplish the proof of Theorem \ref{thm:mass-concentration-boundary}. First, we establish the following delicate estimate of $e(a)$ as $a\searrow0$
inspired by \cite{Guo-Luo-Zhang-2018}.

\begin{lemma}\label{lem:boundary-e-a-upper-bounded-eatimate}
Suppose that $V(x)$ satisfies \eqref{eqn:V-potential-x-x-i} and $ Z_1=\emptyset$, then
$$
0\leq e(a)\leq \kappa^{\frac{2}{p+2}}\left(\frac{p+4}{2(p+2)}\right)^{\frac{2p}{p+2}}
\Big[\Big(\frac{p}{2}\Big)^{\frac{2}{p+2}}+\Big(\frac{2}{p}\Big)^{\frac{p}{p+2}}\Big]a^{\frac{p}{p+2}}
\Big(\ln\frac{1}{a}\Big)^{\frac{2p}{p+2}}\ \mbox{as}\ a\searrow0,
$$
where $p>0$ and $\kappa>0$ are defined by \eqref{eqn:p-definition} and \eqref{eqn:k-lambda-definition}, respectively.
\end{lemma}

\begin{proof}
Due to $ Z_1=\emptyset$, we know that $ Z_0\neq\emptyset$. Then we can choose $x_i\in Z_0$ such that $p_i=p$. According to the inner ball condition of $\Omega$, there exists an open ball $B_R(x_0)$ such that $x_i\in\partial\Omega\cap\partial B_R(x_0)$, where $x_0$ is an inner point of $\Omega$ and $R>0$ is sufficiently small. Let
$R_\tau:=\frac{g(\tau)}{\tau}<R$ for $\tau>0$ large enough, where $0<g(\tau)\in C^2(\mathbb R)$ is to be determined later and satisfies $\lim_{\tau\rightarrow\infty}g(\tau)=\infty$ and $\lim_{\tau\rightarrow\infty}\frac{g(\tau)}{\tau}=0$. Take $x_\tau=x_i-(1+\eta(\tau))R_\tau\overrightarrow{n}$, where $\overrightarrow{n}$ is the unit outward normal vector to $\partial\Omega$ at the point $x_i$ and $\eta(\tau)>0$ satisfies $\eta(\tau)\rightarrow0$ as $\tau\rightarrow\infty$ is also to be determined later. Then we can directly obtain $B_{(1+\eta(\tau))R_\tau}(x_\tau)\subset\Omega$ and $x_i\in\partial\Omega\cap B_{(1+\eta(\tau))R_\tau}(x_\tau)$ satisfies $\lim_{\tau\rightarrow\infty}x_\tau=x_i$.

Let $\varphi_\tau(x)\in C_0^\infty(\mathbb R^3)$ be a nonnegative smooth cut-off function satisfying $\varphi_\tau(x)=1$ for $|x|\leq 1$ and $\varphi_\tau(x)=0$ for $|x|\geq 1+\eta(\tau)$. Without loss of generality, we suppose that $|\nabla\varphi(x)|\leq\frac{M}{\eta(\tau)}$, where $M>0$ is independent of $\tau$. Consider the following trial function
$$
\psi_\tau(x)=\frac{A_\tau\tau^{\frac{3}{2}}}{|Q|_2^2}\varphi_\tau\left(\frac{x-x_\tau}{R_\tau}\right)Q(\tau(x-x_\tau)),\ x\in\Omega,
$$
where $A_\tau>0$ is chosen so that $\int_\Omega|\psi_\tau(x)|^2\mathrm{d}x=1$. Set
$\Omega_\tau=\{x\in\mathbb R^3:(\frac{x}{\tau}+x_\tau)\in\Omega\}$, then the definition of $x_\tau$ implies that $\Omega_\tau\rightarrow\mathbb R^3$ as $\tau\rightarrow\infty$. Moreover, we can compute that
$$
\begin{aligned}
\frac{1}{A_\tau^2}
&=\int_{\Omega_\tau}\frac{1}{|Q|_2^2}\varphi_\tau^2\Big(\frac{x}{\tau R_\tau}\Big)Q^2(x)\mathrm{d}x\\
&=\int_{\mathbb R^3}\frac{1}{|Q|_2^2}\Big(\varphi_\tau^2\Big(\frac{x}{\tau R_\tau}\Big)-1\Big)Q^2(x)\mathrm{d}x+1\\
&\geq 1-Ce^{-2\tau R_\tau}\ \mbox{as}\ \tau\rightarrow\infty,
\end{aligned}
$$
which indicates that
\begin{equation}\label{eqn:A-tau-2-tau-R-tau}
1\leq A_\tau^2\leq1+Ce^{-2\tau R_\tau}\ \mbox{as}\ \tau\rightarrow\infty.
\end{equation}
Meanwhile, we have
$$
\begin{aligned}
&\quad\ \int_{\Omega}|\nabla\psi_\tau|^2\mathrm{d}x\\
&=\int_\Omega\frac{A_\tau^2\tau^3}{|Q|_2^2R_\tau^2}
\Big|\nabla\varphi_\tau\Big(\frac{x-x_\tau}{R_\tau}\Big)\Big|^2Q^2(\tau(x-x_\tau))\mathrm{d}x
+\int_\Omega\frac{A_\tau^2\tau^5}{|Q|_2^2}\varphi_\tau^2\Big(\frac{x-x_\tau}{R_\tau}\Big)|\nabla Q(\tau(x-x_\tau))|^2\mathrm{d}x\\
&\quad\ +2\int_\Omega\frac{A_\tau^2\tau^4}{|Q|_2^2R_\tau}\varphi_\tau\Big(\frac{x-x_\tau}{R_\tau}\Big)Q(\tau(x-x_\tau))
\nabla\varphi_\tau\Big(\frac{x-x_\tau}{R_\tau}\Big)\nabla Q(\tau(x-x_\tau))\mathrm{d}x\\
&:=I_\tau^1+I_\tau^2+I_\tau^3.
\end{aligned}
$$
The direct computations show that
$$
\begin{aligned}
I_\tau^1&=\int_{\mathbb R^3}\frac{A_\tau^2\tau^3R_\tau}{|Q|_2^2}|\nabla\varphi_\tau(x)|^2Q^2(\tau R_\tau x)\mathrm{d}x\\
&=\int_{1\leq|x|\leq1+\eta(\tau)}\frac{A_\tau^2\tau^3R_\tau}{|Q|_2^2}|\nabla\varphi_\tau(x)|^2Q^2(\tau R_\tau x)\mathrm{d}x\\
&\leq\frac{C}{\eta^2(\tau)R_\tau^2}e^{-2\tau R_\tau}\ \mbox{as}\ \tau\rightarrow\infty,
\end{aligned}
$$

$$
\begin{aligned}
I_\tau^2&=\int_{\mathbb R^3}\frac{A_\tau^2\tau^2}{|Q|_2^2}\Big(\varphi_\tau^2\Big(\frac{x}{\tau R_\tau}\Big)-1\Big)|\nabla Q(x)|^2\mathrm{d}x+\int_{\mathbb R^3}\frac{A_\tau^2\tau^2}{|Q|_2^2}|\nabla Q(x)|^2\mathrm{d}x\\
&\leq\int_{\mathbb R^3}\frac{A_\tau^2\tau^2}{|Q|_2^2}|\nabla Q(x)|^2\mathrm{d}x\ \mbox{as}\ \tau\rightarrow\infty
\end{aligned}
$$
and
$$
\begin{aligned}
I_\tau^3&=2\int_{\mathbb R^3}\frac{A_\tau^2\tau^4R_\tau^2}{|Q|_2^2}\varphi_\tau(x)Q(\tau R_\tau x)\nabla\varphi_\tau(x)\nabla Q(\tau R_\tau x)\mathrm{d}x\\
&\leq\frac{2MA_\tau^2\tau^4 R_\tau^2}{\eta(\tau)|Q|_2^2}\int_{1\leq|x|\leq1+\eta(\tau)}Q(\tau R_\tau x)
|\nabla Q(\tau R_\tau x)|\mathrm{d}x\\
&\leq\frac{C\tau}{\eta(\tau)R_\tau}e^{-2\tau R_\tau}\ \mbox{as}\ \tau\rightarrow\infty.
\end{aligned}
$$
Thus, we get that
$$
\quad\ \int_{\Omega}|\nabla\psi_\tau|^2\mathrm{d}x
\leq \int_{\mathbb R^3}\frac{A_\tau^2\tau^2}{|Q|_2^2}|\nabla Q(x)|^2\mathrm{d}x
+C\Big(\frac{\tau}{R_\tau\eta(\tau)}+\frac{1}{\eta^2(\tau)R_\tau^2}\Big)e^{-2\tau R_\tau}\ \mbox{as}\ \tau\rightarrow\infty.
$$

Now we choose a suitable function $ \eta(\tau)>0$ (for example, $ \eta(\tau)=(g(\tau))^{-\alpha}$, $0<\alpha<1$) so that
$$
\frac{\tau}{R_\tau\eta(\tau)}=o(\tau^2)\ \mbox{and}\ \frac{1}{R_\tau^2\eta^2(\tau)}=o(\tau^2)\ \mbox{as}\ \tau\rightarrow\infty,
$$
which, together with \eqref{eqn:A-tau-2-tau-R-tau}, gives that
$$
\int_{\Omega}|\nabla\psi_\tau|^2\mathrm{d}x\leq\tau^2+C\tau^2e^{-2\tau R_\tau}\ \mbox{as}\ \tau\rightarrow\infty.
$$
Then
$$
\left(\int_\Omega|\nabla\psi_\tau|^2\mathrm{d}x\right)^2\leq\tau^4+C\tau^4e^{-2\tau R\tau}\ \mbox{as}\ \tau\rightarrow\infty.
$$
Analogously, we can infer that
$$
\begin{aligned}
\int_\Omega|\psi_\tau|^{2+\frac{8}{3}}\mathrm{d}x
&=\int_{\Omega}\frac{A_\tau^{2+\frac{8}{3}}\tau^7}{|Q|_2^{2+\frac{8}{3}}}
\varphi_\tau^{2+\frac{8}{3}}\Big(\frac{x-x_\tau}{R_\tau}\Big)Q^{2+\frac{8}{3}}(\tau(x-x_\tau))\mathrm{d}x\\
&=\int_{\mathbb R^3}\Bigg(\frac{A_\tau^{2+\frac{8}{3}}\tau^7}{|Q|_2^{2+\frac{8}{3}}}\varphi_\tau^{2+\frac{8}{3}}\Big(\frac{x}{\tau R_\tau}\Big)-1\Bigg)Q^{2+\frac{8}{3}}(x)\mathrm{d}x+\frac{7A_\tau^{2+\frac{8}{3}}\tau^7}{3|Q|_2^{\frac{8}{3}}}\\
&=\int_{|x|\geq1}\frac{A_\tau^{2+\frac{8}{3}}\tau^{7}R_\tau^3}{|Q|_2^{2+\frac{8}{3}}}
(\varphi_\tau^{2+\frac{8}{3}}(x)-1)Q^{2+\frac{8}{3}}
(\tau R_\tau x)\mathrm{d}x+\frac{7A_\tau^{2+\frac{8}{3}}\tau^7}{3|Q|_2^{\frac{8}{3}}}\\
&\geq\frac{7\tau^4}{3|Q|_2^{\frac{8}{3}}}-o(\tau^4e^{-2\tau R_\tau})\ \mbox{as}\ \tau\rightarrow\infty.
\end{aligned}
$$
Consequently, in light of the above estimates, we have
\begin{equation}\label{eqn:nabla-term-estimate}
\begin{aligned}
&\quad\ a\int_\Omega|\nabla\psi_\tau|^2\mathrm{d}x+\frac{b}{2}\left(\int_\Omega|\nabla\psi_\tau|^2\mathrm{d}x\right)^2
-\frac{3\beta^*}{7}\int_\Omega|\psi_\tau|^{2+\frac{8}{3}}\mathrm{d}x\\
&\leq a\tau^2+C\tau^4e^{-2\tau R_\tau}\ \mbox{as}\ \tau\rightarrow\infty.
\end{aligned}
\end{equation}
On the other hand,
since $V(x)$ satisfies \eqref{eqn:V-potential-x-x-i} and $\Omega$ is a bounded domain, then there exists $M>0$ such that
$$
V(x)\leq M|x-x_i|^p\ \mbox{in}\ \Omega,
$$
which indicates that
\begin{equation}\label{eqn:V-M-x-tau-x-i-p}
V(\frac{x}{\tau}+x_\tau)\leq M\Big|\frac{x}{\tau}+x_\tau-x_i\Big|^p\ \mbox{for\ any}\ x\in\Omega_\tau.
\end{equation}
Take
$$
g(\tau)=\tau R_\tau=\frac{p+4}{2}\ln\tau,
$$
then it obviously satisfies the previous setting and indicates that
\begin{equation}\label{eqn:o-ln-tau-tau}
C\tau^4e^{-2\tau R_\tau}=C\tau^{-p}=o\Big(\Big(\frac{\ln\tau}{\tau}\Big)^p\Big)\ \mbox{as}\ \tau\rightarrow\infty.
\end{equation}
Thus, based on the definition of $x_\tau$ and \eqref{eqn:V-M-x-tau-x-i-p},
it follows from the direct calculations that
$$
\begin{aligned}
\int_{\Omega}V(x)|\psi_\tau|^2\mathrm{d}x
&=\int_{\Omega_\tau}\frac{A_\tau^2}{|Q|_2^2}V\Big(\frac{x}{\tau}+x_\tau\Big)
\varphi_\tau^2\Big(\frac{2x}{(p+4)\ln\tau}\Big)Q^2(x)\mathrm{d}x\\
&=\int_{B_{\sqrt{\ln\tau}}(0)}\frac{A_\tau^2}{|Q|_2^2}V\Big(\frac{x}{\tau}+x_\tau\Big)Q^2(x)\mathrm{d}x\\
&\quad\ +\int_{B_{\frac{(1+\eta(\tau))(p+4)\ln\tau}{2}}(0)\backslash B_{\sqrt{\ln\tau}}(0)}\frac{A_\tau^2}{|Q|_2^2}V\Big(\frac{x}{\tau}+x_\tau\Big)
\varphi_\tau^2\Big(\frac{2x}{(p+4)\ln\tau}\Big)Q^2(x)\mathrm{d}x\\
&\leq\int_{B_{\frac{\sqrt{\ln\tau}}{\tau}}(x_\tau)}\frac{A_\tau^2
\tau^3}{|Q|_2^2}V(x)Q^2(\tau(x-x_\tau))\mathrm{d}x\\
&\quad\ +\int_{B_{\frac{(1+\eta(\tau))(p+4)\ln\tau}{2}}(0)\backslash B_{\sqrt{\ln\tau}}(0)}\frac{MA_\tau^2}{|Q|_2^2}
\Big|\frac{x}{\tau}+x_\tau-x_i\Big|^pQ^2(x)\mathrm{d}x\\
&=\kappa_i\Big(\frac{p+4}{2}\Big)^p\Big(\frac{\ln\tau}{\tau}\Big)^p+o\Big(\Big(\frac{\ln\tau}{\tau}\Big)^p\Big)
\ \mbox{as}\ \tau\rightarrow\infty,
\end{aligned}
$$
where $\kappa_i$ is defined in \eqref{eqn:k-i-lambda-idefinition}. Consequently, combining \eqref{eqn:nabla-term-estimate} and \eqref{eqn:o-ln-tau-tau}, we see that
$$
\begin{aligned}
e(a)&\leq a\tau^2+C\tau^4e^{-2\tau R_\tau}+\int_\Omega{V(x)|\psi_\tau|^2\mathrm{d}x}\\
&\leq a\tau^2+\kappa_i\Big(\frac{p+4}{2}\Big)^p\Big(\frac{\ln\tau}{\tau}\Big)^p+o\Big(\Big(\frac{\ln\tau}{\tau}\Big)^p\Big)
\ \mbox{as}\ \tau\rightarrow\infty.
\end{aligned}
$$
Taking
$$
\tau=\Big(\frac{p\kappa_i(p+4)^p}{2^{p+1}}\Big)^{\frac{1}{p+2}}a^{-\frac{1}{p+2}}
(p+2)^{\frac{-p}{p+2}}\Big(\ln\frac{1}{a}\Big)^{\frac{p}{p+2}}
$$
so that $\tau\rightarrow\infty$ as $a\searrow0$,
we further obtain that
\begin{equation}\label{eqn:energy-estimate-functional}
0\leq e(a)\leq \kappa_i^{\frac{2}{p+2}}\left(\frac{p+4}{2(p+2)}\right)^{\frac{2p}{p+2}}
\Big[\Big(\frac{p}{2}\Big)^{\frac{2}{p+2}}+\Big(\frac{2}{p}\Big)^{\frac{p}{p+2}}\Big]a^{\frac{p}{p+2}}
\Big(\ln\frac{1}{a}\Big)^{\frac{2p}{p+2}}\ \mbox{as}\ a\searrow0.
\end{equation}
This means the proof of Lemma \ref{lem:boundary-e-a-upper-bounded-eatimate} is complete.
\end{proof}

\begin{remark}\label{rem:estimate-details}\rm
We provide some details related to the estimate \eqref{eqn:energy-estimate-functional}.
Consider the function
$$
f_r(s)=as^{-2}+rs^p\Big(\ln\frac{1}{s}\Big)^p,\ s\in\Big(0,\frac{1}{e+1}\Big),
$$
where $r>0$, $p>0$ and $a\searrow0$.
Then we have
$$
\begin{aligned}
f_r'(s)&=-2as^{-3}+prs^{p-1}\left(\ln \frac{1}{s}\right)^{p}-prs^{p-1}\left(\ln \frac{1}{s}\right)^{p-1}\\
&=s^{p-1}\left(pr\left[\left(\ln \frac{1}{s}\right)^{p}- \left(\ln \frac{1}{s}\right)^{p-1}\right]-2as^{-(2+p)}\right)\\
&:=s^{p-1}(g(s)-h(s)).
\end{aligned}
$$
By comparing $g(s)$ and $h(s)$ with $a$ small (replace the variable with $t=\frac{1}{s}$), we can conclude that $f_r(s)$ with $s\in\Big(0,\frac{1}{e+1}\Big)$
has a unique minimum point $s_r>0$.
Moreover, let $f_r'(s)=0$, i.e.,
$$
s^{2+p}\left[\left(\ln \frac{1}{s}\right)^{p}- \left(\ln \frac{1}{s}\right)^{p-1}\right]=\frac{2a}{pr},
$$
then we can obtain that the minimum point $s_r$ satisfies
\begin{equation}\label{eqn:minimum-point-estimate}
s_r=(1+o(1))\Big(\frac{2}{pr}\Big)^{\frac{1}{p+2}}a^{\frac{1}{p+2}}(p+2)^{\frac{p}{p+2}}\Big(\ln\frac{1}{a}\Big)^{-\frac{p}{p+2}}
\ \mbox{as}\ a\searrow0
\end{equation}
and the minimum
\begin{equation}\label{eqn:minimum-estimate}
\begin{aligned}
f_r(s_r)&=as_r^{-2}+rs_r^p\Big(\ln\frac{1}{s_r}\Big)^p\\
&=(1+o(1))\Big(\frac{1}{p+2}\Big)^{\frac{2p}{p+2}}\Big[\Big(\frac{p}{2}\Big)^{\frac{2}{p+2}}
+\Big(\frac{2}{p}\Big)^{\frac{p}{p+2}}\Big]r^{\frac{2}{p+2}}a^{\frac{p}{p+2}}
\Big(\ln\frac{1}{a}\Big)^{\frac{2p}{p+2}}\ \mbox{as}\ a\searrow0.
\end{aligned}
\end{equation}
\end{remark}

\begin{lemma}\label{lem:z-a-k-mathcal-Z-0}
Suppose $V(x)$ satisfies \eqref{eqn:V-potential-x-x-i} and $Z_1=\emptyset$, then
for the subsequence $\{a_k\}$ obtained in Lemma {\rm \ref{lem:w-a-k-Q-2-converge-behavior}}, there holds
 $z_{a_k}\rightarrow x_i$ as $k\rightarrow\infty$ for some $x_i\in Z_0$.
\end{lemma}
\begin{proof}
By Lemma \ref{lem:w-a-k-Q-2-converge-behavior}, we get $z_{a_k}\rightarrow x_{i_0}$ as $k\rightarrow\infty$ for some $x_{i_0}\in\bar\Omega$ satisfying $V(x_{i_0})=0$.
Since $ Z_1=\emptyset$,
we first assume that $x_{i_0}$ is an inner point of $\Omega$ such that $0<p_{i_0}<p$.
Then we also know that $\Omega_0=\lim_{k\rightarrow\infty}\Omega_{a_k}=\mathbb R^3$. According to the proof of Lemma \ref{lem:z-a-k-mathcal-Z-1},
whether $\left\{\frac{|z_{a_k}-x_{i_0}|}{\varepsilon_{a_k}}\right\}$ is bounded uniformly as $k\rightarrow\infty$ or $\frac{|z_{a_k}-x_{i_0}|}{\varepsilon_{a_k}}\rightarrow\infty$ as $k\rightarrow\infty$, we can obtain that for enough small $R>0$, there exists $C_0(R)>0$, independent of $a_k$, such that
$$
e(a_k)\geq a_k\varepsilon_{a_k}^{-2}+C_0(R)\varepsilon_{a_k}^{p_{i_0}}
\geq\Big(\left(\frac{p_{i_0}}{2}\right)^{\frac{2}{p_{i_0}+2}}+\left(\frac{2}{p_{i_0}}\right)^{\frac{p_{i_0}}{p_{i_0}+2}}\Big)
C_0(R)^{\frac{2}{p_{i_0}+2}}a_k^{\frac{p_{i_0}}{p_{i_0}+2}}\ \mbox{as}\ a_k\searrow0,
$$
which contradicts Lemma \ref{lem:boundary-e-a-upper-bounded-eatimate} because of $p_{i_0}<p$.

Now it suffices to consider that $x_{i_0}\in\partial\Omega$ satisfies $0<p_{i_0}<p$. Similar to the proof of Lemma \ref{lem:z-a-k-mathcal-Z-1}, we know that \eqref{eqn:contrary-bounded-e-a-V-M} holds if $\frac{|z_{a_k}-x_{i_0}|}{\varepsilon_{a_k}}\rightarrow\infty$ as $k\rightarrow\infty$ or \eqref{eqn:using-guo-prop-3.1-e-a} holds if $\{\frac{|z_{a_k}-x_{i_0}|}{\varepsilon_{a_k}}\}$ is bounded uniformly as $k\rightarrow\infty$, which also contradicts Lemma \ref{lem:boundary-e-a-upper-bounded-eatimate} due to
$p_{i_0}<p$.
Hence, we conclude that $z_{a_k}\rightarrow x_i\in Z_0$ as $k\rightarrow\infty$ and
the proof is complete.
\end{proof}

\begin{lemma}\label{lem:z-a-k-x-i-ln-varepsilon-upper-lower}
Suppose $V(x)$ satisfies \eqref{eqn:V-potential-x-x-i} and $Z_1=\emptyset$, it results that
$$
\limsup\limits_{k\rightarrow\infty}\frac{|z_{a_k}-x_i|}{\varepsilon_{a_k}|\ln\varepsilon_{a_k}|}<\infty\ \mbox{and}\
\liminf\limits_{k\rightarrow\infty}\frac{|z_{a_k}-x_i|}{\varepsilon_{a_k}|\ln\varepsilon_{a_k}|}\geq\frac{p+4}{2}.
$$
\end{lemma}

\begin{proof}
We first prove the former.
Assume to the contrary that there exists a subsequence, still denoted by $\{a_k\}$, such that
$\frac{|z_{a_k}-x_i|}{\varepsilon_{a_k}|\ln\varepsilon_{a_k}|}\rightarrow\infty$ as $k\rightarrow\infty$. This implies
that $\frac{|z_{a_k}-x_i|}{\varepsilon_{a_k}}\rightarrow\infty$ as $k\rightarrow\infty$ and then $\Omega_0=\lim_{k\rightarrow\infty}\Omega_{a_k}=\mathbb R^3$.
Moreover, using Lemma \ref{lem:w-a-k-Q-2-converge-behavior} and Fatou's lemma, we can infer that for any large $M>0$,
$$
\begin{aligned}
&\quad\liminf\limits_{\varepsilon_{a_k}\rightarrow0}\frac{1}{(\varepsilon_{a_k}|\ln\varepsilon_{a_k}|)^p}
\int_{B_{2R}(0)\cap\Omega_{a_k}}V(\varepsilon_{a_k}x+z_{a_k})|\omega_{a_k}|^2\mathrm{d}x&\\
&\geq C\int_{B_{2R}(0)\cap\Omega_{a_k}}\liminf\limits_{\varepsilon_{a_k}\rightarrow0}
\Big|\frac{x}{|\ln\varepsilon_{a_k}|}+\frac{z_{a_k}-x_{i}}{\varepsilon_{a_k}|\ln\varepsilon_{a_k}|}\Big|^{p_{i}}
\prod\limits_{j=1,j\neq i_0}^n|\varepsilon_{a_k}x+z_{a_k}-x_j|^{p_j}|\omega_{a_k}|^2\mathrm{d}x\\
&\geq M,
\end{aligned}
$$
then
$$
e(a_k)\geq a_k\varepsilon_{a_k}^{-2}+M(\varepsilon_{a_k}|\ln\varepsilon_{a_k}|)^p
\geq CM^{\frac{2}{p+2}}a_k^{\frac{p}{p+2}}\Big(\ln\frac{1}{a_k}\Big)
^{\frac{2p}{p+2}}
\ \mbox{as} \ a_k\searrow0,$$
which contradicts Lemma \ref{lem:boundary-e-a-upper-bounded-eatimate}. Thus, the former holds.

Next we verify the remaining one. Assume by contradiction that there exists a subsequence, still denoted by $\{a_k\}$, such that $\frac{|z_{a_k}-x_i|}{\varepsilon_{a_k}|\ln\varepsilon_{a_k}|}\rightarrow\eta<\frac{p+4}{2}$ as $k\rightarrow\infty$.  Then, similar to the proof of
\eqref{eqn:using-guo-prop-3.1-e-a}, we have
\begin{equation}\label{eqn:lower-bounded-z-a-k-ln-varepsilon}
\begin{aligned}
e(a_k)
&\geq\frac{a_k}{\varepsilon_{a_k}^2}\int_{\Omega_{a_k}}|\nabla\omega_{a_k}|^2\mathrm{d}x
+\frac{b}{2\varepsilon_{a_k}^4}\left(\int_{\Omega_{a_k}}|\nabla\omega_{a_k}|^2\mathrm{d}x\right)^2
-\frac{3\beta^*}{7\varepsilon_{a_k}^4}\int_{\Omega_{a_k}}|\omega_{a_k}|^{2+\frac{8}{3}}\mathrm{d}x\\
&\geq\frac{a_k}{\varepsilon_{a_k}^2}
+C\frac{1}{\varepsilon_{a_k}^{4}}e^{-\frac{2}{1+\sigma_k}\frac{|z_{a_k}-x_{i}|}{\varepsilon_{a_k}}}
=a_k\varepsilon_{a_k}^{-2}+C\varepsilon_{a_k}^{\frac{2\eta}{1+\sigma_k}-4}
\ \mbox{as}\ k\rightarrow\infty.
\end{aligned}
\end{equation}
If $\frac{\eta}{1+\sigma_k}\leq2$ for some sufficiently large $k$, then we derive that $e(a_k)\geq C$ for some sufficiently large $k$, which contradicts Lemma \ref{lem:boundary-e-a-upper-bounded-eatimate}. On the other hand, if $\frac{\eta}{1+\sigma_k}>2$ as $k\rightarrow\infty$, then there exists a sufficiently large $k_0$ such that for $k>k_0$,
$$
\eta+\frac{p+4}{2}>\frac{2\eta}{1+\sigma_k}.
$$
Taking $0<r:=\frac{2\eta+p-4}{2}<p$, we further obtain that for $k>k_0$,
$$
0<1-\frac{1}{\frac{\eta}{1+\sigma_k}-1}
<1-\frac{1}{\frac{1}{2}(\eta+\frac{p+4}{2})-1}
=\frac{r}{r+2}<\frac{p}{p+2},
$$
which, together with \eqref{eqn:lower-bounded-z-a-k-ln-varepsilon}, yields that for $k>k_0$,
$$
e(a_k)\geq Ca_k^{1-\frac{1}{\frac{\eta}{1+\sigma_k}-1}}>Ca_k^{\frac{r}{r+2}}.
$$
This still contradicts Lemma \ref{lem:boundary-e-a-upper-bounded-eatimate}. Hence the proof is complete.
\end{proof}

\noindent\textbf{Proof of Theorem \ref{thm:mass-concentration-boundary}.}
By Lemma \ref{lem:z-a-k-x-i-ln-varepsilon-upper-lower}, we see that $\frac{|z_{a_k}-x_i|}{\varepsilon_{a_k}}\rightarrow\infty$ as $k\rightarrow\infty$, which gives $\Omega_0=\lim_{k\rightarrow\infty}\Omega_{a_k}=\mathbb R^3$. It follows from Lemma \ref{lem:w-a-k-Q-2-converge-behavior} that
$$
\lim\limits_{k\rightarrow\infty}\omega_{a_k}=\omega_0=\frac{Q(|x|)}{|Q|_2}\ \mbox{in}\ H^1(\mathbb R^3),
$$
which shows that $\omega_0$ admits the exponential decay as $|x|\rightarrow\infty$. Moreover, by Lemma \ref{lem:w-a-k-Q-2-converge-behavior}-($ii$), we know that $\omega_{a_k}$ admits the exponential decay near $|x|\rightarrow\infty$ as $k\rightarrow\infty$. Then applying the Lemma \ref{lem:z-a-k-x-i-ln-varepsilon-upper-lower} again, we can deduce that
\begin{equation}\label{eqn:in-order-to-V-estimate-ln-varepsilon}
\int_{\Omega_{a_k}}\Big|\frac{x}{|\ln\varepsilon_{a_k}|
}+\frac{z_{a_k}-x_i}{\varepsilon_{a_k}|\ln\varepsilon_{a_k}|}\Big|^p
\Big|\omega_{a_k}^2-\frac{Q^2(x)}{|Q|_2^2}\Big|\mathrm{d}x\rightarrow0\
\mbox{as}\ k\rightarrow\infty.
\end{equation}

Now, we can establish the desired lower bound estimate.
In view of $\Omega_{a_k}=\{x\in\mathbb R^3:\varepsilon_{a_k}x+z_{a_k}\in\Omega\}$, it is clear that $\{\varepsilon_{a_k}x+z_{a_k}\}$ is bounded uniformly in $k$. Then there exists $M>0$ such that
$$
V(\varepsilon_{a_k}x+z_{a_k})\leq M|\varepsilon_{a_k}x+z_{a_k}-x_i|^p,\ x\in\Omega_{a_k}.
$$
Together with \eqref{eqn:in-order-to-V-estimate-ln-varepsilon}, we obtain
$$
\begin{aligned}
&\quad\ \Big|\int_{\Omega_{a_k}}V(\varepsilon_{a_k}x+z_{a_k})\omega_{a_k}^2\mathrm{d}x
-\int_{\Omega_{a_k}}V(\varepsilon_{a_k}x+z_{a_k})\frac{Q^2(x)}{|Q|_2^2}\mathrm{d}x\Big|\\
&\leq M\int_{\Omega_{a_k}}|\varepsilon_{a_k}x+z_{a_k}-x_i|^p\Big|\omega_{a_k}^2-\frac{Q^2(x)}{|Q|_2^2}\Big|\mathrm{d}x\\
&=M\varepsilon_{a_k}^p|\ln\varepsilon_{a_k}|^p
\int_{\Omega_{a_k}}\Big|\frac{x}{|\ln\varepsilon_{a_k}|}
+\frac{z_{a_k}-x_i}{\varepsilon_{a_k}|\ln\varepsilon_{a_k}|}\Big|^p
\Big|\omega_{a_k}^2-\frac{Q^2(x)}{|Q|_2^2}\Big|\mathrm{d}x\\
&=o(\varepsilon_{a_k}^p|\ln\varepsilon_{a_k}|^p)\ \mbox{as}\ k\rightarrow\infty.
\end{aligned}
$$
Consequently, we infer that
$$
\begin{aligned}
e(a_k)&
\geq a_k\varepsilon_{a_k}^{-2}+\int_{\Omega_{a_k}}V(\varepsilon_{a_k}x+z_{a_k})\omega_{a_k}^2\mathrm{d}x\\
&=a_k\varepsilon_{a_k}^{-2}+\int_{\Omega_{a_k}}V(\varepsilon_{a_k}x+z_{a_k})\frac{Q^2(x)}{|Q|_2^2}\mathrm{d}x
+o(\varepsilon_{a_k}^p|\ln\varepsilon_{a_k}|^p)\\
&=a_k\varepsilon_{a_k}^{-2}+\kappa_i\beta_0^p\varepsilon_{a_k}^p|\ln\varepsilon_{a_k}|^p
+o(\varepsilon_{a_k}^p|\ln\varepsilon_{a_k}|^p)\ \mbox{as}\ k\rightarrow\infty,
\end{aligned}
$$
where $\beta_0:
=\lim_{k\rightarrow\infty}\frac{|z_{a_k}-x_i|}{\varepsilon_{a_k}|
\ln\varepsilon_{a_k}|}\geq\frac{p+4}{2}$ due to Lemma \ref{lem:z-a-k-x-i-ln-varepsilon-upper-lower}. Moreover, by \eqref{eqn:minimum-estimate}, we get that
$$
e(a)\geq \kappa^{\frac{2}{p+2}}\left(\frac{\beta_0}{p+2}\right)^{\frac{2p}{p+2}}
\Big[\Big(\frac{p}{2}\Big)^{\frac{2}{p+2}}+\Big(\frac{2}{p}\Big)^{\frac{p}{p+2}}\Big]a^{\frac{p}{p+2}}
\Big(\ln\frac{1}{a}\Big)^{\frac{2p}{p+2}}\ \mbox{as}\ a\searrow0.
$$
Then Lemma \ref{lem:boundary-e-a-upper-bounded-eatimate} implies that $\beta_0\leq\frac{p+4}{2}$ and then $\beta_0=\frac{p+4}{2}$. That is to say, $z_{a_k}$ satisfies
$$
\beta_0=\lim\limits_{k\rightarrow\infty}\frac{|z_{a_k}-x_i|}
{\varepsilon_{a_k}|\ln\varepsilon_{a_k}|}=\frac{p+4}{2}.
$$
Meanwhile, it follows from \eqref{eqn:minimum-point-estimate} that
$$
\lim\limits_{k\rightarrow\infty}\frac{\varepsilon_{a_k}}{a^{\frac{1}{p+2}}\Big(\ln\frac{1}{a}\Big)^{-\frac{p}{p+2}}}
=\Big(\frac{2^{p+1}}{p\kappa(p+4)^p}\Big)^{\frac{1}{p+2}}(p+2)^{\frac{p}{p+2}}.
$$
The proof of Theorem \ref{thm:mass-concentration-boundary} is complete.\qed

\section{Local uniqueness for minimizer}\label{sec:Local uniqueness}

This section is devoted to proving Theorem \ref{thm:local-uniqueness} under the assumptions of $ Z_1=\{x_1\}$ and $p>1$. We will use a contradiction argument as that of \cite[Theorem 1.3]{Guo-2017-SIAM}.
Suppose that there exist two different positive minimizers $u_{1,k}$ and $u_{2,k}$ of $e(a_k)$ with $a_k\searrow0$ as $k\rightarrow\infty$.
Define $\varepsilon_{i,k}:=\left(\int_\Omega|\nabla u_{i,k}|^2\right)^{-\frac{1}{2}}$ (from now on, we abbreviate Lebesgue integral as this form),
then we know that $u_{i,k}(x)$ satisfies the Euler-Lagrange equation with Lagrange multiplier $\mu_{i,k}\in\mathbb R$
\begin{equation*}
-(a_k+b\varepsilon_{i,k}^{-2})\Delta u_{i,k}+V(x)u_{i,k}=\mu_{i,k} u_{i,k}+\beta^*u_{i,k}^{1+\frac{8}{3}}\  \mbox{in}\ {\Omega}.
\end{equation*}
Moreover, let $z_{i,k}$ be the unique maximum point of $u_{i,k}$
and $\bar u_{i,k}(x):=\varepsilon_k^{\frac{3}{2}}u_{i,k}(x)$ with
$\varepsilon_k:=\lambda^{-1}a_k^{\frac{1}{p+2}}$,
then it follows from Theorem \ref{thm:mass-concentration-inner-point} that
\begin{equation}\label{eqn:z-i-k-cover-x-1}
\frac{z_{i,k}-x_1}
{\varepsilon_k}\rightarrow 0\ \mbox{as}\ k\rightarrow\infty
\end{equation}
and
\begin{equation}\label{eqn:tilde-u-conver-Q}
\tilde u_{i,k}:= \bar u_{i,k}(\varepsilon_kx+z_{2,k})\rightarrow \frac{Q(|x|)}
{|Q|_2}
\end{equation}
uniformly in $\mathbb R^3$ as $k\rightarrow\infty$.
Moreover,  $\bar u_{i,k}(x)$ satisfies the equation
\begin{equation}\label{eqn:bar-u-i-k-equation}
-(a_k+b\varepsilon_{i,k}^{-2})\varepsilon_k^4\Delta \bar u_{i,k}+\varepsilon_k^4V(x)\bar u_{i,k}=\mu_{i,k}\varepsilon_k^4\bar u_{i,k}+\beta^*\bar u_{i,k}^{1+\frac{8}{3}}\  \mbox{in}\ {\Omega}.
\end{equation}

Since $u_{1,k}\not\equiv u_{2,k}$, we first consider
$$
\bar\xi_k(x):=\frac{u_{2,k}- u_{1,k}}{\|u_{2,k}-u_{1,k}\|_{L^{\infty}
(\Omega)}}=\frac{\bar u_{2,k}-\bar u_{1,k}}{\|\bar u_{2,k}-\bar u_{1,k}\|_{L^{\infty}
(\Omega)}}
$$
and give a prior estimate for $\bar\xi_k(x)$ motivated by \cite{Guo-2017-SIAM,Cao-2015-CV}.
\begin{lemma}
For any $x_0\in\Omega$, there exists a constant $ \delta>0$ small and $C>0$ such that
\begin{equation}\label{eqn:boundary-estimate-xi-k}
b\varepsilon_k^{2}
\int_{\partial B_{\delta(x_0)}}|\nabla\bar\xi_k|^2 \mathrm dS+\varepsilon_k^4\int_{\partial B_{\delta(x_0)}}V(x)\bar\xi_k^2\mathrm dS
+\frac{b}{6}\int_{\partial B_{\delta(x_0)}}\bar\xi_k^2\mathrm dS
\leq C\varepsilon_k^{3}.
\end{equation}
\end{lemma}
\begin{proof}
In view of \eqref{eqn:bar-u-i-k-equation}, we can infer that
$\bar\xi_k(x)$ satisfies the equations
\begin{equation}\label{eqn:bar-xi-k-equation-1}\begin{aligned}
&\ \ -(a_k+b\varepsilon_{2,k}^{-2})\varepsilon_k^4\Delta \bar\xi_k-\frac{
b(\varepsilon_{2,k}^{-2}-\varepsilon_{1,k}^{-2})\varepsilon_k^4\Delta \bar u_{1,k}}{\|\bar u_{2,k}-\bar u_{1,k}\|_{L^{\infty}(\Omega)}}
+\varepsilon_k^4V(x)\bar\xi_k\\
&=\mu_{1,k}\varepsilon_k^4 \bar\xi_k+
\frac{(\mu_{2,k}-\mu_{1,k})\varepsilon_k^4 \bar u_{2,k}}{\|\bar u_{2,k}-\bar u_{1,k}\|_{L^{\infty}(\Omega)}}+
\beta^*\frac{\bar u_{2,k}^{1+\frac{8}{3}}-\bar u_{1,k}^{1+\frac{8}{3}}}{\|\bar u_{2,k}-\bar u_{1,k}\|_{L^{\infty}(\Omega)}}\  \mbox{in}\ {\Omega}
\end{aligned}
\end{equation}
and
\begin{equation*}\begin{aligned}
&\ \ -(a_k+b\varepsilon_{1,k}^{-2})\varepsilon_k^4\Delta \bar\xi_k-\frac{
b(\varepsilon_{2,k}^{-2}-\varepsilon_{1,k}^{-2})\varepsilon_k^4
\Delta \bar u_{2,k}}{\|\bar u_{2,k}-\bar u_{1,k}\|_{L^{\infty}(\Omega)}}
+\varepsilon_k^4V(x)\bar\xi_k\\
&=\mu_{1,k}\varepsilon_k^4 \bar\xi_k+
\frac{(\mu_{2,k}-\mu_{1,k})\varepsilon_k^4 \bar u_{2,k}}{\|\bar u_{2,k}-\bar u_{1,k}\|_{L^{\infty}(\Omega)}}+
\beta^*\frac{\bar u_{2,k}^{1+\frac{8}{3}}-\bar u_{1,k}^{1+\frac{8}{3}}}{\|\bar u_{2,k}-\bar u_{1,k}\|_{L^{\infty}(\Omega)}}\  \mbox{in}\ {\Omega}.
\end{aligned}
\end{equation*}
Then there holds
\begin{equation}\label{eqn:bar-xi-k-equation}
\begin{aligned}
&\ \ -(2a_k+b\varepsilon_{1,k}^{-2}+b\varepsilon_{2,k}^{-2})
\varepsilon_k^4\Delta \bar\xi_k-\frac{
b(\varepsilon_{2,k}^{-2}-\varepsilon_{1,k}^{-2})\varepsilon_k^4
\Delta (\bar u_{1,k}+\bar u_{2,k})}{\|\bar u_{2,k}-\bar u_{1,k}\|_{L^{\infty}(\Omega)}}
+2\varepsilon_k^4V(x)\bar\xi_k\\
&=2\mu_{1,k}\varepsilon_k^4 \bar\xi_k+
2\frac{(\mu_{2,k}-\mu_{1,k})\varepsilon_k^4 \bar u_{2,k}}{\|\bar u_{2,k}-\bar u_{1,k}\|_{L^{\infty}(\Omega)}}+
2\beta^*\frac{\bar u_{2,k}^{1+\frac{8}{3}}-\bar u_{1,k}^{1+\frac{8}{3}}}{\|\bar u_{2,k}-\bar u_{1,k}\|_{L^{\infty}(\Omega)}}\  \mbox{in}\ {\Omega}.
\end{aligned}
\end{equation}
Note that
\begin{equation}\label{eqn:power-term-xi-theta}
\frac{\bar u_{2,k}^{1+\frac{8}{3}}-\bar u_{1,k}^{1+\frac{8}{3}}}{\|\bar u_{2,k}-\bar u_{1,k}\|_{L^{\infty}(\Omega)}}=(1+\frac{8}{3})\bar\xi_k (\theta\bar u_{2,k}+(1-\theta\bar u_{1,k}))^{\frac{8}{3}}
\end{equation}
for some $\theta\in (0,1)$ and $\|\xi_k\|_{L^{\infty}(\Omega)}=1$, then it follows from \eqref{eqn:tilde-u-conver-Q} that
\begin{equation}\label{eqn:bar-Xi-K-power}
\begin{aligned}
\beta^*\int_{\Omega}\frac{(\bar u_{2,k}^{1+\frac{8}{3}}-\bar u_{1,k}^{1+\frac{8}{3}})\bar\xi_k}{\|\bar u_{2,k}-\bar u_{1,k}\|_{L^{\infty}(\Omega)}}
&=\int_{\Omega}(1+\frac{8}{3})\bar\xi_k^2 (\theta\bar u_{2,k}+(1-\theta\bar u_{1,k}))^{\frac{8}{3}}\\
&\leq \varepsilon_k^3\int_{\Omega_k}(1+\frac{8}{3}) (\theta\tilde u_{2,k}+(1-\theta\tilde u_{1,k}))^{\frac{8}{3}}\\
&\leq C\varepsilon_k^3,
\end{aligned}
\end{equation}
where $\Omega_k:=\{x\in\mathbb R^3:\varepsilon_kx+z_{2,k}\in \Omega\}\rightarrow\mathbb R^3$ as $k\rightarrow\infty$.
Moreover, since
$$
\mu_{i,k}=2e(a_k)-a_k\int_{\Omega}|\nabla u_{i,k}|^2-\frac{\beta^*}{7}\int_{\Omega}|u_{i,k}|
^{2+\frac{8}{3}}-\int_\Omega V(x)u_{i,k}^2,
$$
we have
$$\begin{aligned}
\frac{\mu_{2,k}-\mu_{1,k}}{\|\bar u_{2,k}-\bar u_{1,k}\|_{L^{\infty}(\Omega)}}&=
-a_k\varepsilon_k^{-3}\int_\Omega \nabla (\bar u_{1,k}+\bar u_{2,k})\nabla\bar\xi_k-
\varepsilon_k^{-3}\int_\Omega V(x)(\bar u_{2,k}+\bar u_{1,k})\bar \xi_k\\
&\ \ -\frac{\beta^*}{7}(2+\frac{8}{3})
\varepsilon_k^{-7}
\int_\Omega \bar\xi_k (\theta\bar u_{2,k}+(1-\theta\bar u_{1,k}))^{1+\frac{8}{3}},
\end{aligned}
$$
where
$$
a_k\varepsilon_k^{-3}\int_\Omega \nabla (\bar u_{1,k}+\bar u_{2,k})\nabla\bar\xi_k
\leq \frac{1}{2}a_k \left(\varepsilon_k^{-2}\int_{\Omega_k} |\nabla (\tilde u_{1,k}+\tilde u_{2,k})|^2+\varepsilon_k^{-3}\int_\Omega |\nabla \bar\xi_k|^2\right),
$$
$$
\varepsilon_k^{-3}\int_\Omega V(x)(\bar u_{2,k}+\bar u_{1,k})\bar \xi_k\leq C\int_{\Omega_k}(\tilde u_{2,k}+\tilde u_{1,k})
$$
and
$$
\frac{4}{7}\beta^*(2+\frac{8}{3})\varepsilon_k^{-7}
\int_\Omega \bar\xi_k (\theta\bar u_{2,k}+(1-\theta\bar u_{1,k}))^{1+\frac{8}{3}}\leq C\varepsilon_k^{-4}\int_{\Omega_k} (\theta\tilde u_{2,k}+(1-\theta\tilde u_{1,k}))^{1+\frac{8}{3}}.
$$
Thus we can derive that
\begin{equation}\label{eqn:mu-2-mu-1-estimate}
\frac{(\mu_{2,k}-\mu_{1,k})\varepsilon_k^4}{\|\bar u_{2,k}-\bar u_{1,k}\|_{L^{\infty}(\Omega)}}\leq C(1+a_k\varepsilon_k^{4}\int_{\Omega}|\nabla\bar\xi_k|^2).
\end{equation}
Then due to the fact that
$$
\int_{\Omega}\bar u_{2,k}\bar\xi_k\leq\varepsilon_k^3\int_{\Omega_k}\tilde u_{2,k},
$$
we have
\begin{equation}\label{eqn:bar-mu-Xi-K-power}
\frac{(\mu_{2,k}-\mu_{1,k})\varepsilon_k^4}{\|\bar u_{2,k}-\bar u_{1,k}\|_{L^{\infty}(\Omega)}}\int_{\Omega}\bar u_{2,k}\bar\xi_k
\leq C\varepsilon_k^3+
C a_k\varepsilon_k^{4}
\int_{\Omega}|\nabla\bar\xi_k|^2.
\end{equation}
In addition, the direct calculation gives that
\begin{equation}\label{eqn:Delta-bar-u-i}
\begin{aligned}
-\frac{
(\varepsilon_{2,k}^{-2}-\varepsilon_{1,k}^{-2})\varepsilon_k^4}{\|\bar u_{2,k}-\bar u_{1,k}\|_{L^{\infty}(\Omega)}}
\int_{\Omega}
\Delta (\bar u_{1,k}+\bar u_{2,k}) \bar\xi_k
=\varepsilon_k\left(\int_{\Omega}\nabla (\bar u_{1,k}+\bar u_{2,k})\nabla\bar\xi_k\right)^2\geq0.
\end{aligned}
\end{equation}

Consequently, multiplying \eqref{eqn:bar-xi-k-equation} by $\bar\xi_k$ and integrating over $\Omega$, we can infer from \eqref{eqn:bar-Xi-K-power}-\eqref{eqn:Delta-bar-u-i} that
$$
\begin{aligned}
&\ \ (2a_k+b\varepsilon_k^{-2}\int_{\Omega_k}|\nabla \tilde u_{1,k}|^2+|\nabla \tilde u_{2,k}|^2)\varepsilon_k^4
\int_{\Omega}|\nabla\bar\xi_k|^2
+2\varepsilon_k^4\int_{\Omega}V(x)\bar\xi_k^2
-2\mu_{1,k}\varepsilon_k^4\int_{\Omega}\bar\xi_k^2\\
&\leq C\varepsilon_k^3+
C a_k\varepsilon_k^{4}
\int_{\Omega}|\nabla\bar\xi_k|^2.
\end{aligned}
$$
Utilizing \eqref{eqn:nabla-omega-1} and \eqref{eqn:mu-a-estimate},
we know that
\begin{equation}\label{eqn:mu-varepsilon-cover-property}
\int_{\Omega_k}|\nabla \tilde u_{2,k}|^2\rightarrow1\
\mbox{and}\
\mu_{i,k}\varepsilon_k^4\rightarrow-\frac{b}{6}\ \mbox{as}\ k\rightarrow\infty,
\end{equation}
then we deduce that
\begin{equation}\label{eqn:estimate-xi-k-R-N}
b\varepsilon_k^{2}
\int_{\Omega}|\nabla\bar\xi_k|^2+\varepsilon_k^4\int_{\Omega}V(x)\bar\xi_k^2
+\frac{b}{6}\int_{\Omega}\bar\xi_k^2
\leq C\varepsilon_k^{3}\ \mbox{as}\ k\rightarrow\infty.
\end{equation}
Hence, by \cite[Lemma 4.5]{Cao-2015-CV}, we conclude that \eqref{eqn:boundary-estimate-xi-k} holds.
\end{proof}
In the sequel, we define
$$
\xi_k(x):=\bar \xi_k(\varepsilon_kx+z_{2,k})
$$
and provide the proof of Theorem \ref{thm:local-uniqueness} by
deriving a desired contradiction.

\noindent\textbf{Proof of Theorem \ref{thm:local-uniqueness}.}
The proof is divided into the following four steps.

{ \bf\emph {Step} 1.}
There exist some constants $b_0,\bar b_0$ and $b_i$ $(i=1,2,3)$ such that
$ \xi_k(x)\rightarrow \xi_0(x)$ in $C_{loc}^1(\mathbb R^3)$ as $k\rightarrow\infty$, where
\begin{equation}\label{eqn:xi-0}
\xi_0=b_0 Q+\bar b_0x\cdot\nabla Q+\sum_{i=1}^3b_i\frac{\partial  Q}{\partial x_i}.
\end{equation}

Taking advantage of \eqref{eqn:bar-xi-k-equation-1}, \eqref{eqn:power-term-xi-theta} and the definition of $\tilde u_{i,k}(x)$, we know that $\xi_k(x)$ satisfies the equation
\begin{equation*}
\begin{aligned}
&\ \ -(a_k\varepsilon_k^2+b\int_{\Omega_k}|\nabla \tilde u_{2,k}|^2)\Delta \xi_k-\frac{
b(\varepsilon_{2,k}^{-2}-\varepsilon_{1,k}^{-2})\varepsilon_k^2
\Delta \tilde u_{1,k}}{\|\bar u_{2,k}-\bar u_{1,k}\|_{L^{\infty}(\Omega)}}
+\varepsilon_k^4V(\varepsilon_kx+z_{2,k})\xi_k\\
&=\mu_{1,k}\varepsilon_k^4 \xi_k+
\frac{(\mu_{2,k}-\mu_{1,k})\varepsilon_k^4 \tilde u_{2,k}}{\|\bar u_{2,k}-\bar u_{1,k}\|_{L^{\infty}(\Omega)}}
+\beta^*(1+\frac{8}{3})\xi_k (\theta\tilde u_{2,k}+(1-\theta\tilde u_{1,k}))^{\frac{8}{3}}\  \mbox{in}\ {\Omega_k}.
\end{aligned}
\end{equation*}
It can be simply written as
\begin{equation}\label{eqn:xi-k-equation}
\begin{aligned}
-(a_k\varepsilon_k^2+b\int_{\Omega_k}|\nabla \tilde u_{2,k}|^2)\Delta \xi_k+C_k(x)\xi_k=g_k(x) \  \mbox{in}\ {\Omega_k},
\end{aligned}
\end{equation}
where
$$
C_k(x)=\varepsilon_k^4V(\varepsilon_kx+z_{2,k})-\mu_{i,k}
\varepsilon_k^4 -
\beta^*(1+\frac{8}{3}) (\theta\tilde u_{2,k}+(1-\theta\tilde u_{1,k}))^{\frac{8}{3}}
$$
and
$$\begin{aligned}
g_k(x)&=-\frac{
b(\varepsilon_{1,k}^{-2}-\varepsilon_{2,k}^{-2})\varepsilon_k^2
\Delta \tilde u_{1,k}}{\|\bar u_{2,k}-\bar u_{1,k}\|_{L^{\infty}(\Omega)}}
+\frac{(\mu_{2,k}-\mu_{1,k})\varepsilon_k^4 \tilde u_{2,k}}{\|\bar u_{2,k}-\bar u_{1,k}\|_{L^{\infty}(\Omega)}}\\
&=\left(\int_{\Omega_k}\nabla (\tilde u_{1,k}+\tilde u_{2,k})\nabla\xi_k\right)\Delta \tilde u_{1,k}
+\frac{b}{2}\left(\frac{\varepsilon_{2,k}^{-4}
-\varepsilon_{1,k}^{-4}}{\|\bar u_{2,k}-\bar u_{1,k}\|_{L^{\infty}(\Omega)}}\right)\varepsilon_k^4 \tilde u_{2,k}\\
&\ \ -\frac{4}{7}\beta^*(2+\frac{8}{3})\left(
\int_{\Omega_k} \xi_k (\theta\tilde u_{2,k}+(1-\theta\tilde u_{1,k}))^{1+\frac{8}{3}}\right)\tilde u_{2,k}\\
&=\left(\int_{\Omega_k}\nabla (\tilde u_{1,k}+\tilde u_{2,k})\nabla\xi_k\right)\Delta \tilde u_{1,k}
-\frac{4}{7}\beta^*(2+\frac{8}{3})\left(
\int_{\Omega_k} \xi_k (\theta\tilde u_{2,k}+(1-\theta\tilde u_{1,k}))^{1+\frac{8}{3}}\right)\tilde u_{2,k}\\
&\ \ +\frac{b}{2}\left(\int_{\Omega_k} (|\nabla \tilde u_{1,k}|^2+|\nabla \tilde u_{2,k}|^2)\int_{\Omega_k} \nabla (\tilde u_{1,k}+\tilde u_{2,k})\nabla\xi_k\right)\tilde u_{2,k}.
\end{aligned}
$$
Here we have used the fact that
$$
\mu_{i,k}=e(a_k)+\frac{b}{2}\varepsilon_{i,k}^{-4}-\frac{4}{7}
\beta^*\int_{\Omega}u_{i,k}^{2+\frac{8}{3}}.
$$
From \eqref{eqn:estimate-xi-k-R-N}, we know that $\int_{\Omega_k}|\nabla \xi_k|^2\leq C$ as $k\rightarrow\infty$. Then
due to $\|\xi_k\|_{L^\infty(\Omega_k)}\leq1$, the standard elliptic regularity implies that $\|\xi_k\|_{C_{loc}^{1,\alpha}(\Omega_k)}\leq C$ for some $\alpha\in (0,1)$ and constant $C>0$ is independent of $k$. Therefore, the exists a subsequence of $\{a_k\}$ and a function
$\xi_0=\xi_0(x)$ such that $ \xi_k(x)\rightarrow \xi_0(x)$ in $C_{loc}^1(\mathbb R^3)$ as $k\rightarrow\infty$.
By \eqref{eqn:tilde-u-conver-Q}, \eqref{eqn:mu-varepsilon-cover-property} and the definition of $\beta^*$, we know that
$$
C_k(x)\rightarrow \frac{b}{6}-(1+\frac{8}{3})\beta^* Q^{\frac{8}{3}}|Q|_2^{-\frac{8}{3}}=\frac{b}{6}
-\frac{b}{2}(1+\frac{8}{3})Q^{\frac{8}{3}}
$$
and
$$\begin{aligned}
g_k(x)&\rightarrow 2|Q|_2^{-2}\left(\int_{\mathbb R^N}\nabla Q\nabla\xi_0\right)\Delta Q
-\frac{4}{7}\beta^*(2+\frac{8}{3})|Q|_2^{-(2+\frac{8}{3})}\left(
\int_{\mathbb R^N} \xi_0 Q^{1+\frac{8}{3}}\right)Q\\
&\ \ +4|Q|_2^{-2}\left(|Q|_2^{-2}\int_{\mathbb R^N} |\nabla Q|^2\int_{\mathbb R^N} \nabla Q\nabla\xi_0\right)Q\\
&=2|Q|_2^{-2}\left(\int_{\mathbb R^N}\nabla Q\nabla\xi_0\right)\Delta Q
+4|Q|_2^{-2}\left(\int_{\mathbb R^N} \nabla Q\nabla\xi_0\right)Q\\
&\ \ -\frac{2b}{7}(2+\frac{8}{3})|Q|_2^{-2}\left(
\int_{\mathbb R^N} \xi_0 Q^{1+\frac{8}{3}}\right)Q\ \mbox{in}\
C_{loc}^1(\mathbb R^3)\  \mbox{as}\  k\rightarrow\infty.
\end{aligned}
$$
Thus, we can conclude that $\xi_0$ satisfies
$$
\begin{aligned}
&\ \ -2 \Delta \xi_0+\frac{1}{3}\xi_0-
(1+\frac{8}{3})Q^{\frac{8}{3}}\xi_0\\
&=\frac{4}{b}|Q|_2^{-2}\left(\int_{\mathbb R^N}\nabla Q\nabla\xi_0\right)\Delta Q
+\frac{8}{b}|Q|_2^{-2}\left(\int_{\mathbb R^N} \nabla Q\nabla\xi_0\right)Q
-\frac{4}{7}(2+\frac{8}{3})|Q|_2^{-2}\left(
\int_{\mathbb R^N} \xi_0 Q^{1+\frac{8}{3}}\right)Q\\
&:=\sigma_1 \Delta Q+2\sigma_1 Q-\sigma_2 Q.
\end{aligned}
$$
Define the operator $L:=-2\Delta+\frac{1}{3}-(1+\frac{8}{3})Q^{\frac{8}{3}}$.
Note that
$$
 L(x\cdot\nabla Q)=-4\Delta Q\ \mbox{and}\ L(Q+\frac{4}{3}x\cdot\nabla Q)=-\frac{8}{9} Q.
$$
Then we have that
\begin{equation}\label{eqn:xi-0-definition-1}
\begin{aligned}
L\xi_0&=-\frac{1}{4}\sigma_1L(x\cdot\nabla Q)
-\frac{9}{4}\sigma_1 L(Q+\frac{4}{3}x\cdot\nabla Q)+
\frac{9}{8}\sigma_2 L(Q+\frac{4}{3}x\cdot\nabla Q)\\
&=\left(-\frac{13}{4}\sigma_1+\frac{3}{2}\sigma_2\right) L(x\cdot\nabla Q)+\left(-\frac{9}{4}\sigma_1+\frac{9}{8}\sigma_2\right) L Q.
\end{aligned}
\end{equation}
It is well known that the unique  positive radial solution $Q$ of equation $-2\Delta u+\frac{1}{3}u-|u|^{\frac{8}{3}}u=0$
is non-degenerate, see for example \cite{Oh-1990}. This means that
$$
Ker L=span\left\{\frac{\partial  Q}{\partial x_1},\frac{\partial  Q}{\partial x_2},\frac{\partial  Q}{\partial x_3}\right\}.
$$
Combining with \eqref{eqn:xi-0-definition-1}, we conclude that there exist some constants $b_0,\bar b_0$ and $ b_i$ $(i=1,2,3)$ such that
\begin{equation*}
\xi_0=b_0 Q+\bar b_0x\cdot\nabla Q+\sum_{i=1}^3b_i\frac{\partial  Q}{\partial x_i},
\end{equation*}
that is, \eqref{eqn:xi-0} holds. In particular,
\begin{equation}\label{eqn:b-0-bar-b-0-definition}
b_0=\left(-\frac{9}{4}\sigma_1+\frac{9}{8}\sigma_2\right) \  \mbox{and}\ \bar b_0=\left(-\frac{13}{4}\sigma_1+\frac{3}{2}\sigma_2\right).
\end{equation}

{ \bf\emph {Step} 2.} The local Pohozaev identity holds
\begin{equation}\label{eqn:local-Pohozaev}
\bar b_0\int_{\mathbb R^3}\frac{\partial|x|^p}{\partial x_j} (x\cdot\nabla Q^2)=\sum_{i=1}^3\int_{\mathbb R^3}b_i\frac{\partial^2  |x|^p}{\partial x_j\partial x_i}Q^2.
\end{equation}

Multiplying \eqref{eqn:bar-u-i-k-equation} by
$\frac{\partial u_{i,k}}{\partial x_j}$ and
integrating over a small ball $B_\delta(z_{2,k})\subset \Omega$, we deduce that
\begin{equation}\label{eqn:bar-u-i-k-equation-int}\small
\begin{aligned}
&\ \ -(a_k+b\varepsilon_{i,k}^{-2})\varepsilon_k^4
\int_{B_\delta(z_{2,k})}\Delta \bar u_{i,k}\frac{\partial u_{i,k}}{\partial x_j}
+\varepsilon_k^4\int_{B_\delta(z_{2,k})}V(x)\bar u_{i,k}\frac{\partial u_{i,k}}{\partial x_j}\\
&=\mu_{i,k}\varepsilon_k^4\int_{B_\delta(z_{2,k})}\bar u_{i,k}\frac{\partial u_{i,k}}{\partial x_j}+\beta^*\int_{B_\delta(z_{2,k})}\bar u_{i,k}^{1+\frac{8}{3}}\frac{\partial u_{i,k}}{\partial x_j}\\
&=\frac{1}{2}\mu_{i,k}\varepsilon_k^4\int_{\partial B_\delta(z_{2,k})}\bar u_{i,k}^2\nu_j\mathrm dS+\frac{3\beta^*}{14}\int_{B_\delta(z_{2,k})}\bar u_{i,k}^{2+\frac{8}{3}}\nu_j\mathrm dS,
\end{aligned}
\end{equation}
where $\nu=(\nu_1,\nu_2,\nu_3)$ denotes the outward unit normal of $B_\delta(z_{2,k})$.
Note that
\begin{equation*}\small
\begin{aligned}
&\ \ -(a_k+b\varepsilon_{i,k}^{-2})\varepsilon_k^4
\int_{B_\delta(z_{2,k})}\Delta \bar u_{i,k}\frac{\partial u_{i,k}}{\partial x_j}\\
&=-(a_k+b\varepsilon_{i,k}^{-2})\varepsilon_k^4\left( \int_{\partial B_\delta\left(z_{2, k}\right)} \frac{\partial \bar{u}_{i, k}}{\partial x_j} \frac{\partial \bar{u}_{i, k}}{\partial \nu} \mathrm dS+
\int_{ B_\delta\left(z_{2, k}\right)}\nabla \bar{u}_{i, k}\nabla \frac{\partial \bar{u}_{i, k}}{\partial x_j}\right)\\
&=-(a_k+b\varepsilon_{i,k}^{-2})\varepsilon_k^4\left( \int_{\partial B_\delta\left(z_{2, k}\right)} \frac{\partial \bar{u}_{i, k}}{\partial x_j} \frac{\partial \bar{u}_{i, k}}{\partial \nu} \mathrm dS-\frac{1}{2} \int_{\partial B_\delta\left(z_{2, k}\right)}\left|\nabla \bar{u}_{i, k}\right|^2 \nu_j \mathrm dS\right)
\end{aligned}
\end{equation*}
and
$$
\varepsilon_k^4 \int_{B_\delta\left(z_{2, k}\right)} V(x) \frac{\partial \bar{u}_{i, k}}{\partial x_j} \bar{u}_{i, k}=\frac{\varepsilon_k^4}{2} \int_{\partial B_\delta\left(z_{2, k}\right)} V(x) \bar{u}_{i, k}^2 \nu_j \mathrm d S-\frac{\varepsilon_k^4}{2} \int_{B_\delta\left(z_{2, k}\right)} \frac{\partial V(x)}{\partial x_j} \bar{u}_{i, k}^2.
$$
Then from \eqref{eqn:bar-u-i-k-equation-int}, we have
$$
\begin{aligned}
&\ \ \varepsilon_k^4 \int_{B_\delta\left(z_{2, k}\right)} \frac{\partial V(x)}{\partial x_j} \bar{u}_{i, k}^2\\
&=-2(a_k+b\varepsilon_{i,k}^{-2})\varepsilon_k^4\left( \int_{\partial B_\delta\left(z_{2, k}\right)} \frac{\partial \bar{u}_{i, k}}{\partial x_j} \frac{\partial \bar{u}_{i, k}}{\partial \nu} \mathrm dS-\frac{1}{2} \int_{\partial B_\delta\left(z_{2, k}\right)}\left|\nabla \bar{u}_{i, k}\right|^2 \nu_j \mathrm dS\right)\\
&\ \ -\mu_{i,k}\varepsilon_k^4\int_{\partial B_\delta(z_{2,k})}\bar u_{i,k}^2\nu_j\mathrm dS-\frac{3\beta^*}{7}\int_{B_\delta(z_{2,k})}\bar u_{i,k}^{2+\frac{8}{3}}\nu_j\mathrm dS
+\frac{\varepsilon_k^4}{2} \int_{\partial B_\delta\left(z_{2, k}\right)} V(x) \bar{u}_{i, k}^2 \nu_j \mathrm d S.
\end{aligned}
$$

It follows that
$$
\begin{aligned}
&\ \ \varepsilon_k^4 \int_{B_\delta\left(z_{2, k}\right)} \frac{\partial V(x)}{\partial x_j} (\bar{u}_{i, k}+\bar{u}_{i, k})\bar \xi_k\\
&=-2(a_k\varepsilon_k^2+b\varepsilon_k^2\varepsilon_{2,k}^{-2})\varepsilon_k^2\Big[ \int_{\partial B_\delta\left(z_{2, k}\right)} \frac{\partial \bar{u}_{2, k}}{\partial x_j} \frac{\partial \bar{\xi_k}}{\partial \nu} \mathrm dS+\int_{\partial B_\delta\left(z_{2, k}\right)}\frac{\partial \bar{\xi_k}}{\partial x_j} \frac{\partial \bar{u}_{1, k}}{\partial \nu} \mathrm dS\\
&\ \ -\frac{1}{2} \int_{\partial B_\delta\left(z_{2, k}\right)}(\nabla \bar{u}_{1, k}+\bar{u}_{2, k})\nabla\bar\xi_k \nu_j \mathrm dS\Big]\\
&\ \ -\frac{
b(\varepsilon_{2,k}^{-2}-\varepsilon_{1,k}^{-2})\varepsilon_k^4}{\|\bar u_{2,k}-\bar u_{1,k}\|_{L^{\infty}(\Omega)}}\left( \int_{\partial B_\delta\left(z_{2, k}\right)} \frac{\partial \bar{u}_{1, k}}{\partial x_j} \frac{\partial \bar{u}_{1, k}}{\partial \nu} \mathrm dS-\frac{1}{2} \int_{\partial B_\delta\left(z_{2, k}\right)}\left|\nabla \bar{u}_{1, k}\right|^2 \nu_j \mathrm dS\right)\\
&\ \ +\Big[-\mu_{1,k}\varepsilon_k^4\int_{\partial B_\delta(z_{2,k})}(\bar u_{1,k}+\bar u_{2,k})\bar \xi_k\nu_j\mathrm dS-\frac{(\mu_{2,k}-\mu_{1,k})\varepsilon_k^4}{\|\bar u_{2,k}-\bar u_{1,k}\|_{L^{\infty}(\Omega)}}\int_{\partial B_\delta(z_{2,k})}\bar u_{2,k}^2\nu_j\mathrm dS\\
&\ \ -\frac{3\beta^*}{7}(2+\frac{8}{3})
\int_\Omega \bar\xi_k (\theta\bar u_{2,k}+(1-\theta\bar u_{1,k}))^{1+\frac{8}{3}}+\frac{\varepsilon_k^4}{2} \int_{\partial B_\delta\left(z_{2, k}\right)} V(x) (\bar{u}_{1, k}+\bar{u}_{2, k}) \xi_k^2\nu_j \mathrm d S\Big]\\
&:=I_1+I_2+I_3
\end{aligned}
$$

In the following, we estimate $I_1$, $I_2$ and $I_3$.
Note that
$$
\varepsilon_{2,k}^{-2}=\varepsilon_k^{-2}\int_{\Omega_k}|\nabla \tilde u_{2,k}|^2, \ B_{\frac{\delta}{\varepsilon_k}}(0)\subset \Omega_k
$$
and
$$
\begin{aligned}
&\ \  \varepsilon_k^2 \int_{\partial B_\delta\left(z_{2, k}\right)}\left|\frac{\partial \bar{u}_{2, k}}{\partial x_j} \frac{\partial \bar{\xi}_k}{\partial \nu}\right| d S \\
&\leq \varepsilon_k\left(\int_{\partial B_\delta\left(z_{2, k}\right)}\left|\frac{\partial \bar{u}_{2, k}}{\partial x_j}\right|^2 d S\right)^{\frac{1}{2}}\left(\varepsilon_k^2 \int_{\partial B_\delta\left(z_{2, k}\right)}\left|\frac{\partial \bar{\xi}_k}{\partial \nu}\right|^2 d S\right)^{\frac{1}{2}}.
\end{aligned}
$$
Repeat this inequality for the remaining terms in $I_1$, and then we can use \eqref{eqn:boundary-estimate-xi-k} and the exponential decay in Lemma \ref{lem:w-a-k-Q-2-converge-behavior}-$(ii)$ to obtain that
$$
I_1\leq C\varepsilon_k^2e^{-\frac{C\delta}{\varepsilon_k}}
\ \mbox{as}\ k\rightarrow\infty.
$$
On the other hand, by the definition of $\varepsilon_{i,k}$ and \eqref{eqn:estimate-xi-k-R-N}, we know that
\begin{equation*}
\begin{aligned}
\frac{
(\varepsilon_{2,k}^{-2}-\varepsilon_{1,k}^{-2})\varepsilon_k^4}{\|\bar u_{2,k}-\bar u_{1,k}\|_{L^{\infty}(\Omega)}}
&=\varepsilon_k\int_{\Omega}\nabla ( \bar u_{1,k}+ \bar u_{2,k})\nabla\bar\xi_k\\
&\leq\left(\int_{\Omega}|\nabla (\bar u_{1,k}+\bar u_{2,k})|^2\right)^{\frac{1}{2}}
\left(\varepsilon_k^2\int_{\Omega}|\nabla\bar\xi_k|
^2\right)^{\frac{1}{2}}\\
&\leq C\varepsilon_k\ \mbox{as}\ k\rightarrow\infty.
\end{aligned}
\end{equation*}
Then we can also get that
$$
I_2\leq C\varepsilon_k^2e^{-\frac{C\delta}{\varepsilon_k}}
\ \mbox{as}\ k\rightarrow\infty.
$$
Similarly, in view of \eqref{eqn:mu-2-mu-1-estimate} and the fact that $|\bar\xi_k|$ are bounded uniformly in $k$, we can determine that
$$
I_3=o(e^{-\frac{C\delta}{\varepsilon_k}})\ \mbox{as}\ k\rightarrow\infty.
$$

Consequently, we have
\begin{equation}\label{eqn:partial-V-estimate}
o(e^{-\frac{C\delta}{\varepsilon_k}})=\varepsilon_k^4 \int_{B_\delta\left(z_{2, k}\right)} \frac{\partial V(x)}{\partial x_j} (\bar{u}_{i, k}+\bar{u}_{i, k})\bar \xi_k \ \mbox{as}\ k\rightarrow\infty.
\end{equation}
It follows from \eqref{eqn:V-potential-x-x-i} and $p_1=p$ that
$$
V(x)=h(x)\prod_{i=1}^n|x-x_i|^{p_i}=|x-x_1|^ph(x)\prod_{i=2}
^n|x-x_i|^{p_i}.
$$
Without loss of generality, we assume that
$$
\lim_{x\rightarrow x_1}h(x)\prod_{i=2}^n|x-x_i|^{p_i}=1.
$$
Then we can deduce that
$$
V(x)=|x-x_1|^p(1+o(1)),\ \frac{\partial V(y)}{\partial y_j}=(1+o(1))\frac{\partial |y-x_1|^p}{\partial y_j}
$$
and
\begin{equation}\label{eqn:V-nabla-V-relation}
V(x)+\frac{1}{2}x\cdot \nabla V(x)=(1+o(1))\frac{2+p}{2} |x-x_1|^p\ \mbox{as}\
x\rightarrow x_1.
\end{equation}
Thus, for small $\delta>0$, we can infer from \eqref{eqn:partial-V-estimate} that
\begin{equation}\label{eqn:pre-pohozaev-estimate}
\begin{aligned}
o(e^{-\frac{C\delta}{\varepsilon_k}})&=\varepsilon_k^4 \int_{B_\delta\left(z_{2, k}\right)} \frac{\partial V(x)}{\partial x_j} (\bar{u}_{1, k}+\bar{u}_{2, k})\bar \xi_k\\
&=\varepsilon_k^{6+p} \int_{B_{\frac{\delta}{\varepsilon_k}}(0)} (1+o(1))\frac{\partial |x-\frac{z_{2,k}-x_1}{\varepsilon_k}|^p}{\partial x_j} (\tilde{u}_{1, k}+\tilde{u}_{2, k}) \xi_k\ \mbox{as}\ k\rightarrow\infty.
\end{aligned}
\end{equation}
Recall that
$$
\frac{z_{2,k}-x_1}{\varepsilon_k}\rightarrow0\ \mbox{as}\ k\rightarrow \infty,
$$
then it follows from \eqref{eqn:xi-0} and \eqref{eqn:pre-pohozaev-estimate} that
$$
\begin{aligned}
0=2\int_{\mathbb R^3}\frac{\partial|x|^p}{\partial x_j} Q\xi_0
&=2\int_{\mathbb R^3}\frac{\partial|x|^p}{\partial x_j} Q\left[b_0 Q+\bar b_0x\cdot\nabla Q+\sum_{i=1}^3b_i\frac{\partial  Q}{\partial x_i}\right]\\
&=\bar b_0\int_{\mathbb R^3}\frac{\partial|x|^p}{\partial x_j} (x\cdot\nabla Q^2)-\sum_{i=1}^3\int_{\mathbb R^3}b_i\frac{\partial^2  |x|^p}{\partial x_j\partial x_i}Q^2,
\end{aligned}
$$
where we have used the fact $\int_{\mathbb R^3}\frac{\partial|x|^p}{\partial x_j} Q^2=0$. Namely, the local Pohozaev identity \eqref{eqn:local-Pohozaev}
 holds.

{ \bf\emph {Step} 3.} $b_0=\bar b_0=b_i=0,$ $i=1,2,3$.

Multiplying \eqref{eqn:bar-u-i-k-equation} by
$(x-z_{2,k})\cdot \nabla \bar u_{i,k}$ and integrating over $B_\delta(z_{2,k})\subset \Omega$, we know that
\begin{equation}\label{eqn:nabla-bar-u-i-k-equation-int}
\begin{aligned}
&-(a_k\varepsilon_k^2+b\varepsilon_{i,k}^{-2}\varepsilon_k^2)\varepsilon_k^2
\int_{B_\delta(z_{2,k})} \left[(x-z_{2,k})\cdot \nabla \bar u_{i,k}\right] \Delta \bar u_{i,k}\\
&=-\varepsilon_k^4\int_{B_\delta(z_{2,k})}V(x)\bar u_{i,k}\left[(x-z_{2,k})\cdot \nabla \bar u_{i,k}\right]\\
&+\mu_{i,k}\varepsilon_k^4\int_{B_\delta(z_{2,k})}\bar u_{i,k}\left[(x-z_{2,k})\cdot \nabla \bar u_{i,k}\right]+\beta^*\int_{B_\delta(z_{2,k})}\bar u_{i,k}^{1+\frac{8}{3}}\left[(x-z_{2,k})\cdot \nabla \bar u_{i,k}\right].\\
\end{aligned}
\end{equation}
According to the direct calculations, we have
$$
\nabla u \nabla (x\cdot \nabla u)=|\nabla u|^2+\frac{1}{2} x\cdot \nabla (|\nabla u|^2)\ \mbox{and}\ \nabla\cdot (|\nabla u|^2x)=x\cdot \nabla (|\nabla u|^2)+3|\nabla u|^2
$$
for $x\in \mathbb R^3$ and $u:\mathbb R^3\mapsto \mathbb R$.
Then we can deduce that
\begin{equation}\label{eqn:nabla-Delta-bar-u-i-k-1}
\small
\begin{aligned}
&\ \  -\varepsilon_k^2 \int_{B_\delta\left(z_{2, k}\right)}\left[\left(x-z_{2, k}\right) \cdot \nabla \bar{u}_{i, k}\right] \Delta \bar{u}_{i, k} \\
&=-\varepsilon_k^2 \int_{\partial B_\delta\left(z_{2, k}\right)} \frac{\partial \bar{u}_{i, k}}{\partial \nu}\left[\left(x-z_{2, k}\right) \cdot \nabla \bar{u}_{i, k}\right]+\varepsilon_k^2 \int_{B_\delta\left(z_{2, k}\right)} \nabla \bar{u}_{i, k} \nabla\left[\left(x-z_{2, k}\right) \cdot \nabla \bar{u}_{i, k}\right] \\
&=-\varepsilon_k^2 \int_{\partial B_\delta\left(z_{2, k}\right)} \frac{\partial \bar{u}_{i, k}}{\partial \nu}\left[\left(x-z_{2, k}\right) \cdot \nabla \bar{u}_{i, k}\right]+\frac{\varepsilon_k^2}{2} \int_{\partial B_\delta\left(z_{2, k}\right)}\left[\left(x-z_{2, k}\right) \cdot \nu\right]\left|\nabla \bar{u}_{i, k}\right|^2\\
&\ \ -\frac{1}{2}\varepsilon_k^2\int_{B_\delta\left(z_{2, k}\right)} |\nabla \bar{u}_{i, k}|^2.
\end{aligned}
\end{equation}
Multiplying \eqref{eqn:bar-u-i-k-equation} by
$\bar u_{i,k}$ and integrating over $B_\delta(z_{2,k})$, we know that
\begin{equation}
\label{eqn:nabla-Delta-bar-u-i-k-2}
\begin{aligned}
&\ \ -(a_k\varepsilon_k^2+b\varepsilon_{i,k}^{-2}\varepsilon_k^2)\varepsilon_k^2
\int_{B_\delta(z_{2,k})} \bar u_{i,k} \Delta \bar u_{i,k}\\
&=-\left(a_k\varepsilon_k^2+b\varepsilon_{i,k}^{-2}\varepsilon_k^2
\right)\left[
\frac{1}{2}\int_{\partial B_\delta(z_{2,k})}  \nabla \bar u_{i,k}^2\cdot \nu \mathrm dS-
\int_{B_\delta(z_{2,k})} |\nabla \bar u_{i,k}|^2\right]
\\
&=-\varepsilon_k^4\int_{B_\delta(z_{2,k})}V(x)\bar u_{i,k}^2+\mu_{i,k}\varepsilon_k^4\int_{B_\delta(z_{2,k})}\bar u_{i,k}^2+\beta^*\int_{B_\delta(z_{2,k})}\bar u_{i,k}^{2+\frac{8}{3}}.\\
\end{aligned}
\end{equation}
On  the other hand,
\begin{equation}\label{eqn:V-ar-u-i-k-equation-int}
\begin{aligned}
&\ \ -\varepsilon_k^4\int_{B_\delta(z_{2,k})}V(x)\bar u_{i,k}\left[(x-z_{2,k})\cdot \nabla \bar u_{i,k}\right]+\mu_{i,k}\varepsilon_k^4\int_{B_\delta(z_{2,k})}\bar u_{i,k}\left[(x-z_{2,k})\cdot \nabla \bar u_{i,k}\right]\\
&\ \ \ \ +\beta^*\int_{B_\delta(z_{2,k})}\bar u_{i,k}^{1+\frac{8}{3}}\left[(x-z_{2,k})\cdot \nabla \bar u_{i,k}\right]\\
&=-\frac{\varepsilon_k^4}{2} \int_{B_\delta\left(z_{2, k}\right)} \bar{u}_{i, k}^2\left[3\left(\mu_{i, k}-V(x)\right)-\left(x-z_{2, k}\right) \cdot \nabla V(x)\right]\\
&\ \ \ \ +\frac{\varepsilon_k^4}{2} \int_{\partial B_\delta\left(z_{2, k}\right)} \bar{u}_{i, k}^2\left[\mu_{i, k}-V(x)\right]\left(x-z_{2, k}\right) \cdot\nu d S\\
&\ \ \ \ +\frac{3\beta^*}{14}\int_{\partial B_\delta(z_{2,k})}\bar u_{i,k}^{2+\frac{8}{3}}(x-z_{2,k})\cdot \nu-\frac{9\beta^*}{14}\int_{B_\delta(z_{2,k})}\bar u_{i,k}^{2+\frac{8}{3}}.
\end{aligned}
\end{equation}
Then from
\eqref{eqn:nabla-bar-u-i-k-equation-int}-\eqref{eqn:V-ar-u-i-k-equation-int}, we can conclude that
\begin{equation*}\small
\begin{aligned}
&\ \ -(a_k\varepsilon_k^2+b\varepsilon_{i,k}^{-2}\varepsilon_k^2)\varepsilon_k^2
\int_{B_\delta(z_{2,k})} \left[(x-z_{2,k})\cdot \nabla \bar u_{i,k}\right] \Delta \bar u_{i,k}\\
&=(a_k\varepsilon_k^2+b\varepsilon_{i,k}^{-2}\varepsilon_k^2)
\Big(-\varepsilon_k^2 \int_{\partial B_\delta\left(z_{2, k}\right)} \frac{\partial \bar{u}_{i, k}}{\partial \nu}\left[\left(x-z_{2, k}\right) \cdot \nabla \bar{u}_{i, k}\right] \\
&\ \ +\frac{\varepsilon_k^2}{2} \int_{\partial B_\delta\left(z_{2, k}\right)}\left[\left(x-z_{2, k}\right) \cdot \nu\right]\left|\nabla \bar{u}_{i, k}\right|^2\Big)-
(a_k\varepsilon_k^2+b\varepsilon_{i,k}^{-2}\varepsilon_k^2)\Big(
\frac{1}{2}\varepsilon_k^2\int_{B_\delta\left(z_{2, k}\right)} |\nabla \bar{u}_{i, k}|^2\Big)\\
&=(a_k\varepsilon_k^2+b\varepsilon_{i,k}^{-2}\varepsilon_k^2)
\Big(-\varepsilon_k^2 \int_{\partial B_\delta\left(z_{2, k}\right)} \frac{\partial \bar{u}_{i, k}}{\partial \nu}\left[\left(x-z_{2, k}\right) \cdot \nabla \bar{u}_{i, k}\right]\\
&\ \ +\frac{\varepsilon_k^2}{2} \int_{\partial B_\delta\left(z_{2, k}\right)}\left[\left(x-z_{2, k}\right) \cdot \nu\right]\left|\nabla \bar{u}_{i, k}\right|^2\Big)
+\left(a_k\varepsilon_k^2+b\varepsilon_{i,k}^{-2}\varepsilon_k^2
\right)\frac{1}{4}\int_{\partial B_\delta(z_{2,k})}  \nabla \bar u_{i,k}^2\cdot \nu \mathrm dS\\
&\ \ -\frac{1}{2}\left(-\varepsilon_k^4\int_{B_\delta(z_{2,k})}V(x)\bar u_{i,k}^2+\mu_{i,k}\varepsilon_k^4\int_{B_\delta(z_{2,k})}\bar u_{i,k}^2+\beta^*\int_{B_\delta(z_{2,k})}\bar u_{i,k}^{2+\frac{8}{3}}\right)\\
&=-\frac{\varepsilon_k^4}{2} \int_{B_\delta\left(z_{2, k}\right)} \bar{u}_{i, k}^2\left[3\left(\mu_{i, k}-V(x)\right)-\left(x-z_{2, k}\right) \cdot \nabla V(x)\right]\\
&\ \ +\frac{\varepsilon_k^4}{2} \int_{\partial B_\delta\left(z_{2, k}\right)} \bar{u}_{i, k}^2\left[\mu_{i, k}-h(x)\right]\left(x-z_{2, k}\right) \cdot\nu d S\\
&\ \ +\frac{3\beta^*}{14}\int_{\partial B_\delta(z_{2,k})}\bar u_{i,k}^{2+\frac{8}{3}}(x-z_{2,k})\cdot \nu-\frac{9\beta^*}{14}\int_{B_\delta(z_{2,k})}\bar u_{i,k}^{2+\frac{8}{3}}.
\end{aligned}
\end{equation*}
According to the last equality above, we have
\begin{equation}\label{eqn:nabla-bar-u-i-k-equation-int-3}\small
\begin{aligned}
&\ \ (a_k\varepsilon_k^2+b\varepsilon_{i,k}^{-2}\varepsilon_k^2)
\Big(-\varepsilon_k^2 \int_{\partial B_\delta\left(z_{2, k}\right)} \frac{\partial \bar{u}_{i, k}}{\partial \nu}\left[\left(x-z_{2, k}\right) \cdot \nabla \bar{u}_{i, k}\right]\\
&\ \ +\frac{\varepsilon_k^2}{2} \int_{\partial B_\delta\left(z_{2, k}\right)}\left[\left(x-z_{2, k}\right) \cdot \nu\right]\left|\nabla \bar{u}_{i, k}\right|^2\Big) +\frac{1}{4}\left(a_k\varepsilon_k^2+b\varepsilon_{i,k}^{-2}\varepsilon_k^2
\right)\int_{\partial B_\delta(z_{2,k})}  \nabla \bar u_{i,k}^2\cdot \nu \mathrm dS\\
&=-\varepsilon_k^4 \int_{B_\delta\left(z_{2, k}\right)} \bar{u}_{i, k}^2\left(\mu_{i, k}-V(x)\right)-\frac{\beta^*}{7}\int_{B_\delta(z_{2,k})}\bar u_{i,k}^{2+\frac{8}{3}}\\
&\ \ +\frac{\varepsilon_k^4}{2} \int_{B_\delta\left(z_{2, k}\right)} \bar{u}_{i, k}^2\left[\left(x-z_{2, k}\right) \cdot \nabla V(x)\right]+\frac{\varepsilon_k^4}{2} \int_{\partial B_\delta\left(z_{2, k}\right)} \bar{u}_{i, k}^2\left[\mu_{i, k}-h(x)\right]\left(x-z_{2, k}\right) \cdot\nu d S\\
&\ \ +\frac{3\beta^*}{14}\int_{\partial B_\delta(z_{2,k})}\bar u_{i,k}^{2+\frac{8}{3}}(x-z_{2,k})\cdot \nu\mathrm dS.
\end{aligned}
\end{equation}
It can be simply written as
\begin{equation}\label{eqn:nabla-bar-u-i-k-equation-int-4}\small
\begin{aligned}
T_{\partial B_\delta\left(z_{2, k}\right)}^i&=-\varepsilon_k^4 \int_{B_\delta\left(z_{2, k}\right)} \bar{u}_{i, k}^2\left(\mu_{i, k}-V(x)\right)-\frac{\beta^*}{7}\int_{B_\delta(z_{2,k})}\bar u_{i,k}^{2+\frac{8}{3}}\\
&\ \ +\frac{\varepsilon_k^4}{2} \int_{B_\delta\left(z_{2, k}\right)} \bar{u}_{i, k}^2x \cdot \nabla V(x)
-\frac{\varepsilon_k^4}{2} \int_{B_\delta\left(z_{2, k}\right)} \bar{u}_{i, k}^2z_{2, k} \cdot \nabla V(x),
\end{aligned}
\end{equation}
where $T_{\partial B_\delta\left(z_{2, k}\right)}^i$ represents the sum of all boundary integrals in \eqref{eqn:nabla-bar-u-i-k-equation-int-3}.

Recall that
$$
\mu_{i,k}=e(a_k)+\frac{b}{2}\varepsilon_{i,k}^{-4}-\frac{4}{7}
\beta^*\int_{\Omega}u_{i,k}^{2+\frac{8}{3}},
$$
which implies that
$$
\mu_{i,k}\varepsilon_k^4\int_{\Omega}\bar u_{i,k}^2=\varepsilon_k^7e(a_k)+\frac{b}{2}\varepsilon_{i,k}^{-4}
\varepsilon_k^7-\frac{4}{7}
\beta^*\int_{\Omega}\bar u_{i,k}^{2+\frac{8}{3}}.
$$
Thus it follows from \eqref{eqn:nabla-bar-u-i-k-equation-int-4} that
\begin{equation}\label{eqn:T-i-int}
\begin{aligned}
T_{\partial B_\delta\left(z_{2, k}\right)}^i
&=-\varepsilon_k^7e(a_k)-\frac{b}{2}\varepsilon_{i,k}^{-4}
\varepsilon_k^7+\frac{3}{7}
\beta^*\int_{\Omega}\bar u_{i,k}^{2+\frac{8}{3}}+
\varepsilon_k^4\int_{B_\delta\left(z_{2, k}\right)} V(x)\bar u_{i,k}^2\\
&\ \ +\frac{\varepsilon_k^4}{2} \int_{B_\delta\left(z_{2, k}\right)} \bar{u}_{i, k}^2x \cdot \nabla V(x)
-\frac{\varepsilon_k^4}{2} \int_{B_\delta\left(z_{2, k}\right)} \bar{u}_{i, k}^2z_{2, k} \cdot \nabla V(x)
+T^i_{\Omega\backslash B_\delta(z_{2,k})},
\end{aligned}
\end{equation}
where $$T^i_{\Omega\backslash B_\delta(z_{2,k})}:=\varepsilon_k^4\mu_{i, k} \int_{\Omega\backslash B_\delta\left(z_{2, k}\right)} \bar{u}_{i, k}^2+\frac{\beta^*}{7}\int_{\Omega\backslash B_\delta(z_{2,k})}\bar u_{i,k}^{2+\frac{8}{3}}.
$$
Hence, by \eqref{eqn:V-nabla-V-relation} and \eqref{eqn:T-i-int} , we have
$$\small
\begin{aligned}
&\ \ T_{\partial B_\delta\left(z_{2, k}\right)}^2-T_{\partial B_\delta\left(z_{2, k}\right)}^1
-\left(T^2_{\Omega\backslash B_\delta(z_{2,k})}-T^1_{\Omega\backslash B_\delta(z_{2,k})}\right)+
\frac{\varepsilon_k^4}{2} \int_{B_\delta\left(z_{2, k}\right)} z_{2, k} \cdot \nabla V(x)(\bar u_{2,k}-\bar u_{1,k})\bar \xi_k\\
&=\frac{3\beta^*}{7}\int_{\Omega}\left(\bar u_{2,k}^{2+\frac{8}{3}}-\bar u_{1,k}
^{2+\frac{8}{3}}\right)-\frac{b}{2}\varepsilon_k^7
(\varepsilon_{2,k}^{-4}-\varepsilon_{1,k}^{-4})
+\varepsilon_k^4\int_{B_\delta\left(z_{2, k}\right)} \left( V(x)+\frac{1}{2}x\cdot \nabla V(x)\right)(\bar u_{2,k}-\bar u_{1,k})\bar \xi_k\\
&=2\beta^*\varepsilon_k^{-4}\int_{\Omega_k} \xi_k (\theta\tilde u_{2,k}+(1-\theta\tilde u_{1,k}))^{1+\frac{8}{3}}-
\frac{b}{2}\varepsilon_k^{3}
\left(\int_{\Omega_k} (|\nabla \tilde u_{1,k}|^2+|\nabla \tilde u_{2,k}|^2)\int_{\Omega_k} \nabla (\tilde u_{1,k}+\tilde u_{2,k})\nabla\xi_k\right)\\
&\ \ +\varepsilon_k^{7+p}\int_{B_\delta\left(z_{2, k}\right)} (1+o(1)|x-\frac{z_{2,k}-x_1}{\varepsilon_{k}}|(\tilde u_{2,k}-\tilde u_{1,k})\tilde \xi_k.
\end{aligned}
$$
Similar to \emph {Step} 2, we can derive that
$$
T_{\partial B_\delta\left(z_{2, k}\right)}^2-T_{\partial B_\delta\left(z_{2, k}\right)}^1
-\left(T^2_{\Omega\backslash B_\delta(z_{2,k})}-T^1_{\Omega\backslash B_\delta(z_{2,k})}\right)
=o(e^{-\frac{C\delta}{\varepsilon_k}})\ \mbox{as}\ k\rightarrow\infty.
$$
Meanwhile, it follows from \eqref{eqn:pre-pohozaev-estimate} that
$$
\frac{\varepsilon_k^4}{2} \int_{B_\delta\left(z_{2, k}\right)} z_{2, k} \cdot \nabla V(x)(\bar u_{2,k}-\bar u_{1,k})\bar \xi_k=o(e^{-\frac{C\delta}{\varepsilon_k}})\ \mbox{as}\ k\rightarrow\infty.
$$
Thus, we conclude that
\begin{equation}\label{eqn:Q-nabla-Q-xi-0=0}
\int_{\mathbb R^3}Q^{1+\frac{8}{3}} \xi_0=\int_{\mathbb R^3}\nabla Q \nabla \xi_0=0
\end{equation}
and
$$
\int_{\mathbb R^3}|x|^p Q\xi_0=0.
$$
Combining \eqref{eqn:Q-nabla-Q-xi-0=0} and the definitions of $b_0, \bar b_0$ (see \eqref{eqn:b-0-bar-b-0-definition}), we deduce that $b_0=\bar b_0=0$.
Then due to the nondegenerate property \eqref{eqn:nondegenerate}, we can infer from the local Pohozaev identity \eqref{eqn:local-Pohozaev} that $b_0=\bar b_0=b_i=0,$ $i=1,2,3$.

{ \bf\emph {Step} 4.} $\xi_0\equiv0$ cannot occur.

Let $x_k$ be a point satisfying $|\xi_k(x_k)|=\|\xi_k(x)\|_{L^{\infty}(\Omega_k)}=1$. Since $\tilde u_{i,k}$ and $Q$ decay exponentially as $|x|\rightarrow\infty$, then applying the maximum principle to \eqref{eqn:xi-k-equation}, we know that $|x_k|\leq C$ uniformly in $k$. By \emph {Step} 1, it results that
$ \xi_k(x)\rightarrow \xi_0(x)\not\equiv0$ in $C_{loc}^1(\mathbb R^3)$ as $k\rightarrow\infty$.
However, from \emph {Step} 3 and  the definition of $\xi_0$ (see \eqref{eqn:xi-0}), we get that $\xi_0\equiv0$. Hence, the desired contradiction is obtained, which indicates that the pre-assumption
$u_{1,k}\not\equiv u_{2,k}$ is false and the proof of Theorem \ref{thm:local-uniqueness} is complete.
\qed
\ \\
\ \\
\noindent\textbf{\Large Acknowledgements}

\noindent {This research is supported by National Natural Science Foundation of China (No.12371120) and Southwest University graduate research innovation project (No. SWUB24031).}
\vspace{0.8em}

\noindent\textbf{\Large Declaration}

\noindent {{\bf Conflict of interest} The authors do not have conflict of interest.}

\bibliographystyle{elsarticle-num}

\end{document}